\documentclass{llncs}
\usepackage[utf8]{inputenc}
\usepackage{enumerate,xspace,amssymb,amsmath}
\usepackage{lslbasicmath}
\usepackage{hyperref} 
\hypersetup{linkcolor=blue, citecolor=blue, urlcolor=blue, colorlinks=true}

\newcommand{\new}[1]{\emph{#1}}
\newcommand{\rhospe}{\rho_{\rm Spe}}
\newcommand{\rhoken}{\rho_{\rm Ken}}
\newcommand{\I}{{\mathbb I}}
\newcommand{\Edir}{E_{\rm dir}}

\makeatletter
\renewenvironment{proof}[1][\proofname]{\par
\normalfont \topsep6\p@\@plus6\p@\relax
\trivlist
\item\relax
{\itshape
#1\@addpunct{.}}\hspace\labelsep\ignorespaces
}
\makeatother

\begin{document}

\title{Assortativity and bidegree distributions \\ on Bernoulli random graph superpositions}
\author{Mindaugas Bloznelis\inst{1} \and Joona Karjalainen\inst{2} \and Lasse Leskelä\inst{2}}
\institute{
Institute of Informatics, Vilnius University\\
Naugarduko 24, LT-03225 Vilnius, Lithuania 
\and
Department of Mathematics and Systems Analysis \\
School of Science, Aalto University  \\
Otakaari 1, FI-02015 Espoo, Finland
}

\date{\today}
\maketitle
\thispagestyle{plain}
\begin{abstract}
A probabilistic generative network model with $n$ nodes and $m$ overlapping layers is obtained as a superposition of $m$ mutually independent Bernoulli random graphs of varying size and strength.  When $n$ and $m$ are large and of the same order of magnitude, the model admits a sparse limiting regime with a tunable power-law degree distribution and nonvanishing clustering coefficient. In this article we prove an asymptotic formula for the joint degree distribution of adjacent nodes. This yields a simple analytical formula for the model assortativity, and opens up ways to analyze rank correlation coefficients suitable for random graphs with heavy-tailed degree distributions. We also study the effects of power laws on the asymptotic joint degree distributions.
\end{abstract}


\section{Introduction}
\label{sec:Introduction}
\paragraph{Overview and objectives.}
Questions in technology, life sciences, and economics are often related to large systems of nodes connected via pairwise interactions which involve uncertainty due to unpredictable node behavior and missing data.  Such uncertainties have been mathematically modeled and analyzed using random graph models of various complexity, including classical  independently linked and uniform random graphs \cite{Frieze_Karonski_2015}, stochastic block models and inhomogeneous Bernoulli graphs \cite{Abbe_2018_JMLR,Boguna_Pastor-Satorras_2003,Holland_Laskey_Leinhardt_1983}, random graphs with given degree distributions \cite{Bollobas_Janson_Riordan_2007,Molloy_Reed_1998}, and generative models involving  preferential attachment and rewiring mechanisms \cite{Albert_Barabasi_2002,Newman_2003_Structure}. While succeeding to obtain a good fit to degree distributions, most earlier models fail to capture second-order effects related to clustering and transitivity. Random intersection graphs \cite{Ball_Sirl_Trapman_2014,Bloznelis_2013,Britton_Deijfen_Lageras_Lindholm_2008,Godehardt_Jaworski_2001,Karjalainen_VanLeeuwaarden_Leskela_2018,Karonski_Scheinerman_Singer-Cohen_1999}, spatial preferential attachment models \cite{Aiello_etal_2009,Jacob_Morters_2017}, and hyperbolic random geometric graphs \cite{Bode_Fountoulakis_Muller_2015,Kiwi_Mitsche_2019,Krioukov_Papadopoulos_Kitsak_Vahdat_Boguna_2010} have been 
successful in extending the analysis to sparse graph models with tunable global clustering coefficient.
Despite remarkable methodological advances obtained in the aforementioned articles and related literature, most models of sparse random graphs still appear somewhat rigid in what comes to modeling finer second-order properties, such as correlations of the degrees of adjacent nodes \cite{Newman_2002} and 
 degree-dependent clustering coefficients \cite{AngelesSerrano_Boguna_2006_I,fountoulakis2021clustering,Vazquez_Pastor-Satorras_Vespignani_2002}. 

\paragraph{Main contributions.}
This article discusses a mathematical network model recently introduced in \cite{Bloznelis_Leskela_2019-12} which is motivated by the structure of social networks composed of a large number of overlapping communities \cite{Breiger_1974}. The model is generated as a superposition of mutually independent Bernoulli random graphs $G_1,\dots, G_m$ of variable size (number of nodes) and strength (link probability), which can be interpreted as \emph{layers} or \emph{communities}. The node sets of the layers are \emph{random} subsets of the underlying population of $n$ nodes. A key feature of the model is that the layer sizes and layer strengths are assumed to be correlated, which allows for example to model social networks with tunable frequencies of strong small communities and weak large communities.  The main contribution of this article is a rigorous mathematical analysis (Theorem~\ref{the:LimitingModeBidegree}) of the bidegree distribution (joint degree distribution of adjacent nodes) of the model in a limiting regime where the number of nodes $n$ and the number of layers $m$ are large and of the same order of magnitude. 
We note that such a regime admits bidegree distributions with statistically dependent marginals. Moreover, power laws can be introduced by choosing suitable layer types (Theorem \ref{the:LL1}). The bidegree distribution yields compact mathematical formulas for the model assortativity (Theorem~\ref{the:ModelAssortativity}) and rank correlations (Theorem~\ref{the:RankCorrelation}) of the adjacent node degrees. The latter theorem is suitable for modeling dependencies in heavy-tailed models with degrees having unbounded second moments. A proof outline of Theorem \ref{the:LimitingModeBidegree} was presented in the  preliminary version  \cite{10.1007/978-3-030-48478-1_5}. We complete the proof with weaker assumptions in this paper. Theorem \ref{the:LL1} and Lemmas \ref{the:truncation} and \ref{the:minfinity}, as well as the proofs for Lemmas \ref{the:BivariateMomentInequality} and \ref{the:DegreeMomentConvergence}, are new.

\paragraph{Related work.}
Degree distributions, clustering, and percolation analysis of the model is presented in \cite{Bloznelis_Leskela_2019-12}. An analogous model where the node sets of the layers are deterministic has been studied in \cite{Yang_Leskovec_2014} in the context of overlapping community detection. Clustering coefficients and small subgraph frequencies for a special case with constant layer strengths have been analyzed in \cite{Grohn_Karjalainen_Leskela_2019-11-28,Karjalainen_VanLeeuwaarden_Leskela_2018,Karjalainen_Leskela_2017,Petti_Vempala_2018-02}.  In the special case with unit layer strengths, the layers become cliques and the model reduces to the passive random intersection graph introduced in \cite{Godehardt_Jaworski_2001}, with degree and clustering properties analyzed in \cite{Bloznelis_2013,Kurauskas_2015-04}. A network model with similar features has been recently presented in \cite{Vadon_Komjathy_VanDerHofstad_2019}. Assortativity and bidegree distributions have earlier been analyzed in the context of random intersection graph models \cite{Bloznelis_2017,Bloznelis_Jaworski_Kurauskas_2013}, inhomogeneous Bernoulli graphs and their extensions \cite{Boguna_Pastor-Satorras_2003,Mahadevan_Krioukov_Fall_Vahdat_2006,Sadeghi_Rinaldo_2014},
preferential attachment models \cite{Krot_OstroumovaProkhorenkova_2017,VanDerHofstad_Litvak_2014}, and configuration models in \cite{VanDerHofstad_Litvak_2014,VanDerHoorn_Litvak_2014,VanDerHoorn_Litvak_2015}. Extremal properties of bidegree correlations in general graphs have been reported in \cite{Czabarka_etal_2017,VanDerHofstad_Litvak_2014}.

\subsection{Notations}

\paragraph{Sets and numbers.}
The cardinality of a set $A$ is denoted $\abs{A}$.
Ordered pairs are denoted by $(i,j)$, and unordered pairs by  $\{i,j\}$.
Here $1(A)$ is defined to be one when statement $A$ is true, and zero otherwise.
We denote $[n] = \{1,\dots,n\}$ and $\Z_+ = \{0,1,2,\dots\}$.
The falling factorial is denoted $(x)_r = x(x-1) \cdots (x-r+1)$.

\paragraph{Graphs.}
A graph is defined as a pair $G = (V(G), E(G))$ where $V(G)$ is the set of nodes, and $E(G)$ is the set of edges (unordered node pairs). Nodes $i$ and $j$ are called adjacent if  $\{i,j\} \in E(G)$. The set of nodes adjacent to node $i$ is denoted $N_G(i) = \{j \in V(G): \{i,j\} \in E(G)\}$. The degree of $i$ is denoted $\deg_G(i) = \abs{N_G(i)}$. The set of ordered pairs of adjacent nodes is denoted by $\Edir (G) = \{(i,j) \in V(G) \times V(G) : \{i,j \} \in E(G)\}$. 

\paragraph{Probability.}
For a probability measure on a countable space we denote $f(x) = f(\{x\})$ and $\int \phi \, df = \sum \phi(x) f(x)$.
The Dirac measure at $x$ is denoted by $\delta_x$. The binomial distribution is denoted by $\Bin(x,y)(s) = \binom{x}{s} (1-y)^{x-s} y^s$, and the Poisson distribution by $\Poi(\lambda)(s) = e^{-\lambda} \frac{\lambda^s}{s!}$. The product and the convolution of probability measures $f$ and $g$ are denoted by $f \otimes g$ and $f \ast g$, respectively.

\section{Assortativity and bidegree distributions}

\subsection{Empirical quantities}
Let $G$ be a graph with a finite node set and a nonempty link set. Here $G$ is viewed as a nonrandom graph or a fixed sample of a random graph.
The (empirical) \new{degree distribution} of $G$ is a probability measure on $\Z_+$ defined by
\[
 f_G(s)
 \weq \frac{1}{\abs{V(G)}} \sum_{i \in V(G)} 1 \big( \deg_G(i) = s \big),
\]
and represents the probability distribution of the random variable $\deg_G(I)$ where $I$ is a random variable obtained by sampling a node uniformly at random. The (empirical) \new{bidegree distribution} of $G$ with a nonempty link set is a probability measure on $\Z_+^2$ defined by
 \[
 f_G^{(2)}(s,t)  \weq \frac{1}{2 \abs{E(G)}} \sum_{(i,j) \in \Edir (G)} 1 \big( \deg_G(i) = s, \, \deg_G(j) = t \big).
\]
This is the joint probability distribution of the pair $(\deg_G(I), \deg_G(J))$ obtained by sampling $(I,J)$ uniformly at random from the set of all ordered node pairs adjacent in $G$. 
A simple computation shows that both marginals of the bidegree distribution are equal to the size-biased degree distribution
\begin{equation}
\label{eq:sizebias1}
 f_G^*(s)
 = \frac{s f_G(s)}{\sum_t t f_G(t)}.
\end{equation}
The Pearson correlation coefficient of the bidegree distribution is called the (empirical) \new{assortativity} of graph $G$ and can be written as
\[
 \Cor_G(\deg_G(I), \deg_G(J))
 \weq \frac{\sum_{s,t} st f_G^{(2)}(s,t) - \left( \sum_s s f_G^*(s) \right)^2}
 {\sum_s s^2 f_G^*(s) - \left( \sum_s s f_G^*(s) \right)^2}.
\]

\subsection{Model quantities}

Let $G$ be a random graph such that $V(G)$ is nonrandom and finite, and $E(G)$ is nonempty with positive probability.
The \new{model degree distribution} of $G$ is defined by
\begin{equation}
 \label{eq:ModelDegree}
 f(s)
 \weq \pr\big( \deg_G(I) = s \big),
\end{equation}
where $I$ is a random node in $V(G)$, selected uniformly at random and independently of $E(G)$. The \new{model bidegree distribution} is defined by
\begin{equation}
 \label{eq:ModelBidegree}
 f_2(s,t) \weq \pr \Big( \deg_G(I) = s, \, \deg_G(J) = t \, \big| \, (I,J) \in \Edir (G) \Big),
\end{equation}
where $(I,J)$ is an ordered pair of distinct nodes of $V(G)$, selected uniformly at random and independently of $E(G)$. By simple computations one may verify that $f_2(t,s) = f_2(s,t)$, and that both marginals of the model bidegree distribution are equal to the size-biased model degree distribution
\begin{equation}
\label{eq:sizebias2}
f_1^*(s) = \frac{s f(s)}{\sum_t t f(t)}.
\end{equation}
The Pearson correlation coefficient of the model bidegree distribution is called the \new{model assortativity}, and can be written as
\begin{equation}
 \label{eq:ModelAssortativity}
\Cor^*(\deg_G(I), \deg_G(J)) \weq \frac{\E^* \deg_G(I) \deg_G(J) - \left(\E^* \deg_G(I) \right)^2}{\E^* \deg_G(I)^2 - ( \E^* \deg_G(I) )^2},
\end{equation}
where $(I,J)$ is an ordered pair of distinct nodes of $V(G)$ selected uniformly at random as above and  $\E^*$ refers to the conditional expectation given the event $\{ (I,J) \in \Edir (G)\}$. 

\section{Random graph superposition model}
\label{sec:Model}
A multilayer network model $G_n$ with $n$ nodes and $m$ layers of sizes $X_{n,k}$ and strengths $Y_{n,k}$, $k=1,\dots,m$, is defined by: (i) sampling for each $k$ a node set $V(G_{n,k})$ uniformly at random from the subsets of $\{1,\dots,n\}$ of size $X_{n,k}$, (ii) linking each node pair in $V(G_{n,k})$ independently with probability $Y_{n,k}$, and (iii) aggregating the layers $G_{n,k}$ by setting 
\[
 V(G_n) \weq \{1,\dots,n\}
\]
and
\[
 E(G_n) \weq E(G_{n,1}) \cup \cdots \cup E(G_{n,m}).
\]
The layers are assumed to be mutually independent, but the size and strength of a layer may be correlated. Formally, the model is defined by a list
\[
 \Big( (G_{n,1},X_{n,1},Y_{n,1}), \dots, (G_{n,m},X_{n,m},Y_{n,m}) \Big)
\] 
of mutually independent random variables with values in $\cG_n \times \{0,\dots,n\} \times [0,1]$, where $\cG_n$ denotes the set of undirected graphs with node set contained in $\{1,\dots,n\}$. We assume that conditionally on $(X_{n,k},Y_{n,k})$, the probability distribution of $V(G_{n,k})$ is uniform on the subsets of $\{1,\dots,n\}$ of size $X_{n,k}$, and conditionally on $(V(G_{n,k}), X_{n,k}, Y_{n,k})$,
the probability distribution of $E(G_{n,k})$ is such that each node pair of $V(G_{n,k})$ is linked with probability $Y_{n,k}$, independently of other node pairs. We obtain a rich class of generative probabilistic models when we assume that for every $m$ and $n$ the layer types $(X_{n,1},Y_{n,1}), \dots, (X_{n,m}, Y_{n,m})$ are mutually independent and (identically) distributed according to a probability measure $P^{(n)}$ on $\{0,\dots,n\} \times [0,1]$.

A large network is modeled as a sequence of network models of the above type indexed by the number of nodes $n=1,2,\dots$ so that the number of layers $m = m_n$ tends to infinity as $n \to \infty$. To obtain a sparse network admitting tractable limiting formulas with rich expressive power, we shall focus on the sparse parameter regime where $m = \operatorname{\Theta}(n)$ and there exists a probability measure $P$ on $\{0,1,\dots\} \times [0,1]$ which approximates the layer type distribution according to $P^{(n)} \to P$ weakly,
together with the convergence of suitable cross moments $P^{(n)}_{rs} \to P_{rs}$, where we use the shorthand notations
\[
 P^{(n)}_{rs}
 \weq \E\Big( (X_{n,k})_r \, Y_{n,k}^s \Big),
 \qquad
 P_{rs}
 \weq \E\Big( (X)_r Y^s \Big).
\]
Here and below $(X,Y)$ stands for a generic  P-distributed random vector.
\begin{remark}
\label{rem:CrossMoments}
We note that $P^{(n)}_{10} = \E (X_{n,k})$ denotes the expected layer size in the model with scale parameter $n$. To appreciate the relevance of other cross moments, we note that for any graph $R$ with $V(R) \subset \{1,\dots,n\}$, $\abs{V(R)}=r$ and $\abs{E(R)}=s$, the probability that $G_{n,k}$ contains $R$ as subgraph equals $(n)_r^{-1} P^{(n)}_{rs}$, and the expected number of $R$-isomorphic subgraphs in $G_{n,k}$ equals $P^{(n)}_{rs}$ divided by the number of automorphisms of $R$. Especially, the expected numbers of links, 2-stars, and 3-stars contained in any particular layer are given by $\frac12 P^{(n)}_{21}$, $\frac12 P^{(n)}_{32}$, and $\frac16 P^{(n)}_{43}$, respectively. 
\end{remark}

When the number of layers is of the same order as the number of nodes  $\frac{m}{n} \to \mu \in (0,\infty)$, $P^{(n)} \to P$ weakly, and $P^{(n)}_{10} \to P_{10} \in (0,\infty)$, we obtain a sparse network $G_n$ with the model degree distribution \eqref{eq:ModelDegree} converging  weakly \cite{Bloznelis_Leskela_2019-12} to a compound Poisson distribution
\begin{equation}
 \label{eq:LimitingModelDegree}
\bar f_1 \weq \CPoi(\lambda, g)
\end{equation}
with rate parameter $\lambda = \mu P_{10}$ and increment distribution 
\begin{equation}
 \label{eq:LimitingModelDegreeIncrement}
 g(s) 
 \weq \int_{\Z_+ \times [0,1]} \Bin(x-1,y)(s) \, \frac{x P(dx,dy)}{P_{10}},
 \qquad s \in \Z_+.
\end{equation}
In other words, the limiting model degree distribution $\bar f_1$ represents the law of $ \sum_{k=1}^\Lambda H_k$, where $\Lambda,H_1,H_2,\dots$ are mutually independent random integers and such that $\law(\Lambda) = \Poi(\lambda)$ and $\law(H_k) = g$.

\section{Main results}

\subsection{Bidegree distribution}

The result below characterizes the limiting bidegree distribution in the random Bernoulli graph superposition model.
The limiting bidegree distribution can be represented as the joint law of random variables
\begin{equation}
 \label{eq:BidegreeRepresentation}
 \big(D_1^*, D_2^* \big)
 \ = \ \big(1 + D_1 + D'_1, \, 1 + D_2 + D'_2 \big),
\end{equation}
where $D_1$, $D_2$, and $(D_1',D_2')$ are mutually independent and such that $D_1$ and $D_2$ follow the limiting degree distribution $\bar f_1$ defined by \eqref{eq:LimitingModelDegree}.  $D_1'$ and $D_2'$ are defined with the help of an auxiliary random vector $(X',Y')$ taking values in $\mathbb{Z}_+ \times [0,1]$ and having the distribution $\frac{(x)_2 y \, P(dx,dy)}{P_{21}}$.
Namely, given $(X',Y')$ the random variables $D_1'$ and $D_2'$ are conditionally independent and both are $\Bin(X'-2,Y')$-distributed. In order to explain the origin of  $(X',Y')$ we fix a vertex pair $\{i,j\}$ and number $k$. Then  $(X',Y')$ represents the limit (in distribution) of the size and strength of the layer $G_{n,k}$ conditioned on the event that  $\{i,j\}\in E(G_{n,k})$ (in this case we say that the edge 
$\{i,j\}$ is produced by $G_{n,k}$). Typically, an edge $\{i,j\}\in E(G_n)$ is produced by a single layer. The (limiting) numbers of neighbours of $i$ and $j$ produced by  this particular layer are 
represented by $D_1'$ and $D_2'$, while the numbers of neigbours produced by other layers are represented by $D_1$ and $D_2$.

The joint distribution of $(D_1^*, D_2^*)$ defined by \eqref{eq:BidegreeRepresentation} can be written as
\begin{equation}
 \label{eq:Bidegree}
\bar f_{2}  \weq \delta_{(1,1)} \ast (\bar f_1 \otimes \bar f_1) \ast f'_{2}
\end{equation}
where $\ast$ refers to the convolution of probability measures on $\Z_+^2$, and $f'_{2}$ stands for the distribution of $(D_1',D_2')$,
\begin{equation}
 \label{eq:BidegreePrime}
  f'_{2}(s,t) \weq \int_{\Z_+ \times [0,1]} \Bin(x-2,y)(s) \Bin(x-2,y)(t) \frac{(x)_2 y \, P(dx,dy)}{P_{21}}.
\end{equation}
The main intuition behind \eqref{eq:BidegreeRepresentation} and \eqref{eq:Bidegree} is that the overlap of the edges that contribute to $D_1$ and $D_1'$ (or $D_2$ and $D_2'$) is negligible. $D_1$ and $D_2$ represent the numbers of neighbours when the common layer is removed, and these numbers are asymptotically independent. Moreover,  removing the common layer does not affect the asymptotic degree distribution $\bar f_1$. The edge between nodes $I$ and $J$ in the definition \eqref{eq:ModelBidegree} is represented by $\delta_{(1,1)}$.

Theorem~\ref{the:LimitingModeBidegree} below characterizes the large-scale limiting behaviour of the bidegree distribution of $G_n$,
\[
 f_{2,n}(s,t)  \weq \pr\Big( \deg_{G_n}(1) = s, \deg_{G_n}(2) = t \ \big| \ (1,2) \in \Edir (G_n) \Big). 
\]
We note that the probability above is the same as in (\ref{eq:ModelBidegree}). Indeed, the probability distribution of $G_n$ is invariant under permution/relabeling of its nodes. Therefore the random  pair $(I,J)$ can be replaced  in (\ref{eq:ModelBidegree}) by a nonrandom one.

\begin{theorem}
\label{the:LimitingModeBidegree}
Let $m,n \to \infty$. Assume that $\frac{m}{n} \to \mu \in (0,\infty)$,  and that $P^{(n)} \to P$ weakly for some probability measure $P$ on $\Z_+ \times [0,1]$.
\begin{enumerate}[(i)]

\item 
\label{ite:Weak}
If $P^{(n)}_{rs} \to P_{rs} \in (0, \infty)$ for $rs=10,21$, then $f_{2,n} \to \bar f_2$ weakly, where the limiting bidegree distribution $\bar f_2$ is defined by \eqref{eq:Bidegree}. 

\item
\label{ite:Wasserstein}
If in addition, $P^{(n)}_{rs} \to P_{rs} < \infty$ for $rs = 32,43$, then $\int \phi \, d f_{2,n} \to \int \phi \, d \bar f_{2}$
for all $\phi: \Z_+^2 \to \R$ such that $\abs{\phi(x,y)} \le c(1 + x^2 + y^2)$ for some constant $c< \infty$.
\end{enumerate}
\end{theorem} 

The approximation \eqref{ite:Weak} in the weak topology assumes that the mean layer size $P^{(n)}_{10}$ and the mean number of edges $\frac12 P^{(n)}_{21}$ per layer converge to nonzero finite limits. When in addition $P^{(n)}_{32}$ and $P^{(n)}_{43}$ converge to finite limits, we obtain a stronger approximation \eqref{ite:Wasserstein} in the Wasserstein-2 metric \cite[Theorem 6.9]{Villani_2009}), which is used to derive approximations for the model assortativity in Theorem~\ref{the:ModelAssortativity}.

We are particularly interested in network models featuring: (i) tunable power law degree distributions and 
 (ii) tunable frequencies of 
 strong small communities and weak large communities.
 To meet requirement (ii) we let the marginals $X$ and $Y$ of $P$ be negatively correlated. For simplicity we 
 will consider the case where the asymptotic community strength $Y$ is a negative power of the layer size $X$.
 To meet requirement (i) we choose a suitable distribution for 
$X$. 
We mention that the asymptotic power law degree distribution and clustering properties in this setup were shown in   \cite{Bloznelis_Leskela_2019-12}. In Theorem \ref{the:LL1} below we establish the first order asymptotic  of the limiting bidegree distribution ${\bar f}_2(s,t)$ as $s,t\to+\infty$.
%

\begin{theorem}
\label{the:LL1}
Denote by $(D_1^*, D_2^*)$ a random vector with the asymptotic bidegree distribution ${\bar f}_2$. Assume that the limiting layer type distribution equals $P(dx,dy) = p(dx)\delta_{q(x)} (dy)$ with
\begin{gather*}
p(x) \weq (a+o(1))x^{-\alpha} \text{ as } x\to \infty \quad \, \text{and} \, \quad q(x) \weq \min \{1, b x^{-\beta} \}
\end{gather*}
with exponents $\alpha>2$, $\beta \in (0,1)$, $\alpha+\beta>3$ and constants $a,b > 0$. If $\beta=0$, then $q(x)$ is a constant and we require $b<1$.
Then as $t \to \infty$,
\begin{align}
\label{2020-02-24++5}
\pr(D_1^* = t) \weq (1+o(1))c' t^{-\frac{\alpha-2}{1-\beta}},
\end{align}
where $c' \weq ab^{\frac{\alpha-2}{1-\beta}} (1-\beta)^{-1} P_{21}^{-1}$.
Denote $\delta_t=t^{1/2}\ln^4(2+t)$.  For $t_1,t_2\to+\infty$ such that $(t_2-t_1)/\delta_{t_2}\to+\infty$ we have
\begin{gather}
\label{2020-02-24++9}
\pr ( D_1^* = t_1, D_2^* = t_2 ) \weq (1+o(1))c''(t_2-t_1)^{-1-\frac{\alpha-2}{1-\beta}} t_1^{-\frac{\alpha-2}{1-\beta}},
\end{gather}
where $c'' \weq \mu a^2 b^{ \frac{2(\alpha-2)}{1-\beta}} (1-\beta)^{-2} P_{21}^{-1}$.
\end{theorem} 
The intuition behind this result is that the largest contribution to each degree $D_i^* = 1+D_i+D_i'$,
 $i=1,2$, is made by the respective term $D_i'$. In particular, the tail asymptotic (\ref{2020-02-24++5}) is that of the tail of $D_i'$. Furthermore, the terms $D_1'$ and $D_2'$ are strongly correlated for large values $t_1$,$t_2$ and the distribution of 
 $(D_1',D_2')$  concentrates around the diagonal.
Indeed, this can be seen if we extrapolate  (\ref{2020-02-24++9}) to the range  $t_2-t_1=const$. Note, however, that this range is excluded by our technical condition 
$\delta_{t_2}=o(t_2-t_1)$.
\subsection{Assortativity}

The following result provides a formula of the limiting model assortativity which is well defined when the limiting degree distribution has a finite third moment.
In the special case with unit strengths, this formula yields the corresponding result for passive random intersection graphs given in \cite[Theorem 3.1]{Bloznelis_Jaworski_Kurauskas_2013}.

\begin{theorem}
\label{the:ModelAssortativity}
Let $m,n \to \infty$. Assume that $\frac{m}{n} \to \mu \in (0,\infty)$, and that $P^{(n)}_{rs} \to P_{rs} < \infty$ for $rs = 10,21,32,43$, for some probability measure $P$ on $\Z_+ \times [0,1]$ such that $P_{10}, P_{21}>0$. Then the model assortativity 
\eqref{eq:ModelAssortativity} converges according to
\[
 \Cor^*(\deg_{G_n}(I), \deg_{G_n}(J))
 \wto
 \frac{P_{21}( P_{43} + P_{33}) - P_{32}^2}
  {P_{21}(P_{43} + P_{32}) - P_{32}^2 + \mu P_{21}^2 ( P_{21} + P_{32})}.
\]
\end{theorem}
The limiting assortativity is always nonnegative by the following result and the fact that $P_{33} \leq P_{32}$.

\begin{lemma}(Generalizes \cite[Remark 2]{Bloznelis_Jaworski_Kurauskas_2013})
\label{the:BivariateMomentInequality}
For any probability distribution $P$ on $\Z_+ \times [0,1]$,
\[
 P_{32}^2
 \wle P_{21} ( P_{43} + P_{33} ).
\]
\end{lemma}

\subsection{Rank correlations}

Assortativity modeled using Pearson's correlation of the bidegree distribution is ill-behaved for graph models where the limiting degree distribution has an infinite third moment \cite{VanDerHofstad_Litvak_2014}. In such cases, rank correlation coefficients provide a robust alternative \cite{VanDerHofstad_Litvak_2014,VanDerHoorn_Litvak_2014,VanDerHoorn_Litvak_2015}. For a probability measure $f$ on $\R^2$ with nondegenerate marginals, Kendall's rank correlation  \cite{Kruskal_1958,Neslehova_2007} is defined by 

\begin{align*}
 \rhoken(f) &= \Cor\Big( \sgn(U^{(1)}-Z^{(1)}), \, \sgn(U^{(2)}-Z^{(2)}) \Big)
\end{align*}
where $\sgn(x) = 1(x>0) - 1(x<0)$, and $(U^{(1)},U^{(2)})$ and $(Z^{(1)},Z^{(2)})$ are mutually independent and $f$-distributed. Spearman's rank correlation is defined as
\[
 \rhospe(f)
 \weq \Cor \Big(r_1(U^{(1)}), \, r_2(U^{(2)}) \Big),
\]
where $(U^{(1)},U^{(2)})$  is $f$-distributed and $r_i(x) = \frac12( f^{(i)}(-\infty,x) + f^{(i)}(-\infty,x] )$ with $f^{(i)}$ denoting the $i$-th marginal distribution of $f$.
There are several alternative definitions for Spearman's rank correlation corresponding to different tie-breaking conventions \cite{Amerise_Tarsitano_2015}. The above definition agrees with the commonly used mid-rank convention \cite[Theorems 14 and 15]{Neslehova_2007}.

\begin{theorem}
\label{the:RankCorrelation}
Let $m,n\to+\infty$.
Assume that $\frac{m}{n} \to \mu \in (0,\infty)$, and $P^{(n)} \to P$ weakly with $P^{(n)}_{rs} \to P_{rs}$ for $rs=10,21$, where $0 <  P_{rs} < \infty$. Then it holds that
\[
 \rhoken(f_{2,n}) \to \rhoken({\bar f}_2)
 \qquad \text{and} \qquad
 \rhospe(f_{2,n}) \to \rhospe({\bar f}_2),
\]
where the limiting bidegree distribution ${\bar f}_2$ is defined by \eqref{eq:Bidegree}.
\end{theorem}

\section{Discussion}
This article describes degree correlations in a sparse network model introduced in \cite{Bloznelis_Leskela_2019-12}, constructed by a natural superposition mechanism with overlapping layers.  The main contribution is a compact explicit description of the limiting model bidegree distribution (Theorem \ref{the:LimitingModeBidegree}), fully characterized in terms of the limiting joint distribution $P$ of layer sizes and layer strengths, and the limiting ratio $\mu$ of the number of layers and the number of nodes. Some remarks deserve further attention.

(i) 
In this work we have studied the model bidegree distribution, whereas several earlier works \cite{VanDerHofstad_Litvak_2014,VanDerHoorn_Litvak_2014,VanDerHoorn_Litvak_2015} have focused on the convergence 
of the empirical bidegree distribution computed from a fixed random graph sample. Based on analogous studies on ergodic properties of clustering coefficient \cite{Karjalainen_VanLeeuwaarden_Leskela_2018,Karjalainen_Leskela_2017}, we expect that both distributions converge to the same limit under mild regularity assumptions.

(ii) The freedom to tune the limiting joint distribution $P$ of layer sizes and layer strengths yields a rich class of network models. As a concrete example, we studied the case where the layer strength is a deterministic function of layer size so that $Y = q(X)$. If layer sizes follow an approximate power law $\pr(X=x) \propto x^{-\alpha}$, and $q(x) \propto x^{-\beta}$, then the limiting degree distribution follows a power law $\pr(D_1 = t) \propto t^{-\delta}$ with $\delta = 1 + \frac{\alpha-2}{1-\beta}$ \cite{Bloznelis_Leskela_2019-12}.
Because the marginals of the limiting bidegree distribution are size-biased versions of the degree distribution, it follows that the marginals of $f_2$ are power laws with density exponent $\delta -1 = \frac{\alpha-2}{1-\beta}$. The dependence structure of the power-law random variables $D_1^*$ and $D_2^*$ is implicitly captured by \eqref{eq:BidegreeRepresentation}. The same functional form of layer strengths has been also investigated in \cite{Yang_Leskovec_2014} for deterministic layer node sets.

(iii) Fitting the model to real data sets is a problem of future research. A fully nonparametric approach to estimating $P$ appears hard if not impossible, even though currently there are no (positive or negative) theoretical results regarding model identifiability. An alternative approach is to restrict to models where $P = P_\theta$ is parametrized by a small-dimensional parameter $\theta$, and develop estimators of $\theta$ using empirical small subgraph counts.  Recent work in this direction includes \cite{Grohn_Karjalainen_Leskela_2019-11-28,Karjalainen_VanLeeuwaarden_Leskela_2018,Karjalainen_Leskela_2017}
for models with constant layer strength. Model fitting with deterministic (unknown) layer node sets has been studied in \cite{Yang_Leskovec_2014}.

\section{Proofs}

\subsection{Proof of Theorem~\ref{the:LimitingModeBidegree}:\eqref{ite:Weak}}
We start with some auxiliary results. Theorem \ref{T1} (shown in \cite{Bloznelis_Leskela_2019-12}) establishes the asymptotic compound Poisson model degree distribution. Lemma \ref{the:truncation} shows that the degrees of nodes 1 and 2 are asymptotically independent in a model where the layers larger than some $M>0$ are removed. Lemma \ref{the:minfinity} shows that this asymptotic independence still holds when $M$ tends to infinity. 

\begin{theorem}[\cite{Bloznelis_Leskela_2019-12}, Theorem 3.1]\label{T1}
Let $\mu>0$. Let $n,m\to+\infty$. Assume that
$m/n \to \mu$. If

(i) $P^{(n)}$ converges weakly to some probability measure $P$ on $\mathbb{Z}_+ \times [0,1]$ and

(ii) $P^{(n)}_{10}$ converges to some number $P_{10} \in (0, \infty)$,

\noindent
then the model degree distribution of $G_n$ converges weakly to a compound Poisson distribution defined by  
\[
\sum_{k=1}^\Lambda H_k,
\]
where $\Lambda, H_1, H_2, \ldots$ are independent random variables with $\law(\Lambda) = \Poi(\mu P_{10})$ and $\law(H_k) = g$ with the probability mass function
\begin{equation*}
 g(s) 
 \weq \int \Bin(x-1,y)(s) \, \frac{x P(dx,dy)}{P_{10}},
 \qquad s \in \Z_+.
\end{equation*}

\end{theorem}

\begin{lemma}{(Independence with truncated communities.):}
\label{the:truncation}
Assume the conditions of Theorem \ref{T1} hold.  Fix $s,t \in \mathbb{N}$, and $M>0$ such that  $\pr ( X \leq M)  > 0$. Let $K = \{k: X_{n,k} \leq M \}$ and define $\hat G = ([n], E(\hat G))$ with
\[
E(\hat G) = \bigcup_{k \in K} E( G_{n,k}).
\]
Let $\hat d_i = \deg_{\hat G}(i)$. Then, as $n \rightarrow \infty$,
\[
\pr (\hat d_1 = s, \hat d_2 = t) \weq \pr (\hat d_1 = s) \pr(\hat d_2 = t) + o(1).
\]
\end{lemma}

\begin{proof}
We construct $\hat G$ as follows. Let $\hat G_{n,k}$ be independent layers with sizes and strengths
\begin{gather*}
\hat X_{n,k} = X_{n,k} \I(X_{n,k} \leq M), \quad \hat Y_{n,k} = Y_{n,k} \I(X_{n,k} \leq M), \quad k = 1\ldots m.
\end{gather*}
We assume that $n$ is large enough so that $\pr(X_{n,k} \leq M) > 0$ and define the (random) edges as $E(\hat G) = \bigcup_{k=1}^m E(\hat G_{n,k})$. 
Denote by $Z_i$ the layers that contain node $i$:
\begin{gather*}
Z_i = \{ k: \; i \in V(\hat G_{n,k}) \}, \quad i=1,2.
\end{gather*}
By Markov's inequality
\begin{align}
\pr (|Z_1| > \ln m) \wle \frac{\E (|Z_1|)}{\ln m} \weq \pr (1 \in V(\hat G_{n,1}))  \frac{m}{\ln m} 
&\weq \E \Big( \frac{ \hat X_{n,1}}{n} \Big) \frac{m}{\ln m} \nonumber \\ 
&\wle  \frac{mM}{n \ln m}.
\label{eq:markov}
\end{align}
Define the events
\begin{gather*}
{\hat {\cal T}}_{1,2}= \{{\hat d}_1=s\} \cap \{{\hat d}_2=t \} \quad \text{and} \quad {\cal E}=\{Z_1 \cap Z_2 \neq \emptyset \}. 
\end{gather*}
For $z_1 = 1, \ldots, m$ we obtain
\begin{align*}
\pr ({\cal E} \; \vert \; \, |Z_1| = z_1)  &\weq \pr (2 \in \cup_{k \in Z_1} V(\hat G_{n,k}) \; \vert \; \, |Z_1| = z_1)  \\
&\wle z_1 \pr (2 \in V(\hat G_{n,1}) \; \vert \; \, 1 \in V(\hat G_{n,1})) \\
&\weq z_1 \E \binom{n-2}{ \hat X_{n,1}  - 2} \binom{n-1}{\hat X_{n,1}  - 1}^{-1}\\
&\weq z_1 \E \frac{ \hat X_{n,1} -1}{n-1} < \frac{z_1 M}{n-1}, 
\end{align*}
thus,
\begin{align}
\pr ({\cal E}, \; |Z_1| \leq \ln m ) &\wle \pr ({\cal E} \; \vert \; \, |Z_1| \leq \ln m ) \nonumber \\
&\wle \pr (Z_1 \cap Z_2 \neq \emptyset \; \vert \; \, |Z_1| = \lfloor \ln m \rfloor) \nonumber \\
&\wle  \frac{M \ln m}{n-1}.
\label{eq:z1z2lnbound}
\end{align}
Since 
\begin{align*}
\pr ({\hat {\cal T}}_{1,2}) 
 \weq &\pr ({\hat {\cal T}}_{1,2}, \; {\cal E}^c, \; |Z_1| \leq \ln m)  + \pr ({\hat {\cal T}}_{1,2}, \; {\cal E}, \; |Z_1| \leq \ln m)  \\
 &+ \pr ({\hat {\cal T}}_{1,2}, \; |Z_1| > \ln m),
\end{align*}
it follows from \eqref{eq:markov} and \eqref{eq:z1z2lnbound} that
\begin{align}
\pr ({\hat {\cal T}}_{1,2}) \weq \pr ({\hat {\cal T}}_{1,2}, \; {\cal E}^c, \; |Z_1| \leq \ln m) + o(1).
\label{eq:t12approx}
\end{align}
The latter probability is split into mutually exclusive events as
\begin{align}
\pr ({\hat {\cal T}}_{1,2}, \; {\cal E}^c, \; |Z_1| \leq \ln m) = \sum_{|A_1| \leq \floor{\ln m} }\pr ({\hat {\cal T}}_{1,2}, \; {\cal E}^c, \; Z_1 = A_1).
\label{eq:psum}
\end{align}
We now express $\pr ({\hat {\cal T}}_{1,2}, \; {\cal E}^c, \; Z_1 = A_1)$ as a product of two independent events. 
For a set $A \subset [m]$, denote by $G_{A}$ the superposition of $\{\hat G_{n,k}: k \in A \}$. Observe that on the event ${\cal E}^c \cap \{Z_1 = A_1 \}$ the degree of node 2 is determined by $G_{A_1^c}$, and that the event $\{Z_1 = A_1 \}$ equals $\{Z_1 \subset A_1\} \cap \{Z_1 \supset A_1 \}$. With these observations in mind, we rewrite the probability $\pr ({\hat {\cal T}}_{1,2}, \; {\cal E}^c, \; Z_1 = A_1)$ as
\begin{align*}
\pr ( d_{G_{A_1}}(1) = s, \; d_{G_{A_1^c}}(2) = t, \; Z_2 \cap A_1 = \emptyset, \; Z_1 \subset A_1, \; Z_1 \supset A_1),
\end{align*}
where $d_G(i) := \deg_G(i)$. The event $\{ d_{G_{A_1}}(1) = s \} \cap \{ Z_2 \cap A_1 = \emptyset\} \cap \{Z_1 \supset A_1 \}$ only depends on the layers $A_1$, and the event $\{ d_{G_{A_1^c}}(2) = t \} \cap \{ Z_1 \subset A_1 \}$ only depends on the layers $A_1^c$, so that the probability above equals
\begin{align}
\label{eq:psplit}
\pr ( d_{G_{A_1}}(1) = s,  Z_2 \cap A_1 \!= \! \emptyset,  Z_1 \supset A_1) \pr( d_{G_{A_1^c}}(2) = t | Z_1 \subset A_1) \pr(Z_1 \subset A_1).
\end{align}

We now employ a coupling argument to approximate the probability \linebreak $\pr( d_{G_{A_1^c}}(2) = t | Z_1 \subset A_1)$. Introduce a random graph $G^*$, which is a superposition of $m$ i.i.d layers $G^*_{n,k}$ with sizes and strengths  $(X^*_{n,k}, Y^*_{n,k})$ following the distribution 
\begin{gather*}
\pr(X_{n,k}^* = x, Y_{n,k}^* \leq y) \weq \pr(\hat X_{n,k} = x, \hat Y_{n,k} \leq y \; | \; 1 \not\in V(\hat G_{n,k})), \quad k=1,\ldots, m.
\end{gather*}
The nodes of the layers, $V(G^*_{n,k})$, are then chosen uniformly at random from the node set $\{2, \ldots, n \}$, and the edges are generated independently with probabilities  $Y^*_{n,k}$.
Define $G^*_{A_1^c} = \{[n] \setminus \{1\}, \cup_{k\in A_1^c} E(G^*_{n,k})\}$. We approximate $d_{G^*_{A_1^c}}(2)$ by $d_{G^*}(2)$ as follows. The union bound gives
\begin{align*}
\pr (d_{G^*_{A_1^c}}(2) \neq d_{G^*}(2) ) &\leq \pr \Big( 2 \in \bigcup_{k \in A_1} V(G^*_{n,k}) \Big) \\
&\leq (\ln m)\pr( 2 \in V(G^*_{n,1})) \\
&\leq \frac{M \ln m }{n-1},
\end{align*}
and so, 
\begin{align*}
\big| \pr (d_{G^*_{A_1^c}}(2) = r) - \pr( d_{G^*}(2) = r) \big| \leq \frac{M \ln m }{n-1}, \quad \forall r=0,1,\ldots,
\end{align*} 
which is $o(1)$ as $n \to \infty$. 

By construction, for  $k \in A_1^c$ the distribution of $G^*_{n,k}$ equals the conditional distribution of $\hat G_{n,k}$, given the event $\{Z_1 \subset A_1\}$. Thus, $\pr( d_{G_{A_1^c}}(2) = t | Z_1 \subset A_1) \linebreak = \pr (d_{G^*_{A_1^c}}(2) = t)$, and the previous inequality gives
\begin{align*}
\pr( d_{G_{A_1^c}}(2) = t | Z_1 \subset A_1) 
&= \pr( d_{G^*}(2) = t) + o(1).
\end{align*}
We insert this into \eqref{eq:psplit}, which together with \eqref{eq:psum} gives
\begin{align*}
\pr &({\hat {\cal T}}_{1,2}, \; {\cal E}^c, \; |Z_1| \leq \ln m)  \\
&\!= \Big(\pr( d_{G^*}(2)\! = \! t) \sum_{|A_1| \leq \floor{\ln m} } \pr ( d_{G_{A_1}}(1) \!=\! s,  Z_2 \cap A_1 \!=\! \emptyset,  Z_1 \!=\! A_1)\Big) + o(1),
\end{align*} 
where we have used the fact that the events $\{ d_{G_{A_1}}(1) \!=\! s,  Z_2 \cap A_1 \!=\! \emptyset,  Z_1 \supset A_1\} $ and $\{Z_1 \subset A_1\}$ are independent.
By using the mutual exclusivity of the events $\{Z_1 = A_1\}$ and the fact that $d_{G_{A_1}}(1) = \hat d_1$ when $Z_1=A_1$, the above expression simplifies to
\begin{align}
\pr( d_{G^*}(2) = t) \pr ( \hat d_1 = s, \;  Z_2 \cap Z_1 = \emptyset, \; |Z_1| \leq \ln m) + o(1).
\label{eq:mutual}
\end{align}
Returning back to \eqref{eq:t12approx}, this yields
\begin{align*}
\pr({\hat {\cal T}}_{1,2}) &\weq \pr({\hat {\cal T}}_{1,2}, \cE^c, |Z_1| \leq \ln m) + o(1) \\
&\weq \pr( d_{G^*}(2) = t) \pr ( \hat d_1 = s, \;  Z_2 \cap Z_1 = \emptyset, \; |Z_1| \leq \ln m) + o(1) \\
&\weq \pr( d_{G^*}(2) = t) \pr ( \hat d_1 = s) + o(1).
\end{align*}
In the last step we used \eqref{eq:markov} and \eqref{eq:z1z2lnbound}.

We complete the proof by showing that $\pr(d_{G^*}(2) = t)- \pr ( {\hat d}_2 = t)=o(1)$. To this aim we employ Theorem \ref{T1}.  Set $({\hat X}, {\hat Y}) = (X{\mathbb I}(X\le M), Y{\mathbb I}(X\le M))$, where $(X,Y) \sim P$. Note that  $(X^*_1, Y^*_1)$ converges in distribution and in ${\cal L}_1$  to $(\hat X, \hat Y)$: indeed,
\begin{align*}
\pr(X^*_{n,1} = x, Y^*_{n,1} \leq y) &= \pr(\hat X_{n,1} = x, \hat Y_{n,1} \leq y \; | \; 1 \not\in V(\hat G_{n,1})) \\
&= \frac{\pr(1 \not\in V(\hat G_{n,1}) | \hat X_{n,1} \!=\! x, \hat Y_{n,1} \!\leq\! y ) \pr ( \hat X_{n,1} \!=\! x, \hat Y_{n,1} \!\leq\! y)}{\pr(1 \not\in V(\hat G_{n,1}) )}\\
&= \frac{\binom{n-1}{x}\binom{n}{x}^{-1} \pr ( \hat X_{n,1} = x, \hat Y_{n,1} \leq y)}{\E \binom{n-1}{\hat X_{n,1}} \binom{n}{\hat X_{n,1}}^{-1}} \\
&=  \frac{1- n^{-1} x}{1- n^{-1} \E  \hat X_{n,1} } \pr ( \hat X_{n,1} = x, \hat Y_{n,1} \leq y) \\
& \to \pr(\hat X = x, \hat Y \leq y).
\end{align*} 
Thus, by Theorem \ref{T1},  $d_{G^*}(2)$ converges to the same distribution as $\hat d_2$, and so
\begin{align}
\pr ({\hat {\cal T}}_{1,2}) = \pr( \hat d_1 = s)\pr ( \hat d_2 = t) + o(1).
\label{eq:independencehat}
\end{align} 
\qed
\end{proof}

\begin{lemma}{(M to infinity.)}
\label{the:minfinity}
Assume $(i)$ and $(ii)$ of Theorem \ref{T1}, and that $m/n \to \mu$. Define $\hat d_1, \hat d_2$ as in Lemma \ref{the:truncation}. Denote $D_i = D_{n,i} = \deg_{G_n}(i)$. As $M \rightarrow \infty$,
\begin{equation}
\label{2021-05-24}
{\sup}_n | \pr (D_1 = s, D_2 = t) - \pr (\hat d_1 = s, \hat d_2 = t)| \rightarrow 0,
\end{equation}
and especially, as $n \rightarrow \infty$,
\begin{equation}
\label{eq:minfinity}
\pr (D_1 = s, D_2 = t) - \pr(D_1=s)\pr(D_2=t)
\rightarrow 0.
\end{equation}
\end{lemma}

\begin{proof}
From the assumptions it follows that $X_{n,1}$, $n=1,\ldots$ are uniformly integrable, hence
\begin{equation}
\label{eq:unifinteg}
\lim_{M\to+\infty} \sup_n  \E \left(X_{n,1}  \I (X_{n,1}>M) \right)=0.
\end{equation}
Denote ${\cal T}_{1,2} = \{D_1 = s, D_2 = t\}$. We approximate the probability of this event by the corresponding probability in the truncated model:
\begin{align}
\label{eq:bidegapprox}
| \pr ({\cal T}_{1,2}) - \pr (\hat d_1 = s, \hat d_2 = t) | &\wle \pr (D_1 \neq \hat d_1) + \pr (D_2 \neq \hat d_2).
\end{align}
The event $\{ D_1 \neq \hat d_1\}$ occurs when removing the layers larger than $M$ also removes an edge between node 1 and another node. Since this can only happen if there exists a layer $i$ such that $1 \in V(G_{n,i})$ and $X_{n,i} > M$,
the union bound gives
\begin{align*}
\pr (D_1 \neq \hat d_1) &\wle \pr \Big( \bigcup_{i=1}^m \{1 \in V(G_{n,i}), X_{n,i}  > M \} \Big) \\
&\wle \sum_{i=1}^m \pr (1 \in V(G_{n,i}), X_{n,i} > M),
\end{align*}
and by the law of total probability
\begin{align*}
\pr (1 \in V(G_{n,i}), X_{n,i} > M ) &\weq \E [\pr (1 \in V(G_{n,i}), X_{n,i} > M \,| \, X_{n,i} ] \\
&\weq \E [\pr (1 \in V(G_{n,i}) \,|\, X_{n,i}) \, \I (X_{n,i}>M) ] \\
&\weq \E \Big( \frac{X_{n,i}}{n} \I (X_{n,i}>M) \Big).
\end{align*}
Putting this back to \eqref{eq:bidegapprox} gives
\[
| \pr ({\cal T}_{1,2}) - \pr (\hat d_1 = s, \hat d_2 = t) | \wle 2 \frac{m}{n} \E \Big( X_{n,1}  \I (  X_{n,1} > M)\Big).
\]
We similarly obtain 
\begin{equation}
\label{eq:pointmassconvergence}
\sup_{k\geq 0} | \pr (D_1 = k) - \pr (\hat d_1 = k) | \wle \pr(D_1 \neq \hat d_1) \wle \frac{m}{n} \E \Big( X_{n,1} \I(X_{n,1} > M) \Big).
\end{equation} 
It follows from \eqref{eq:unifinteg} and the boundedness of $m/n$ that for any $\varepsilon>0$ there exists $M_\varepsilon$ such that $\forall M > M_\varepsilon$
\begin{align}
| \pr ({\cal T}_{1,2}) - \pr (\hat d_1 = s, \hat d_2 = t) | &\wle \varepsilon, \quad \forall s,t \geq 0, 
\nonumber \\
| \pr (D_1 = k) - \pr (\hat d_1 = k) | &\wle  \varepsilon, \quad \forall k \geq 0. \label{eq:d1tohatd1}
\end{align} 
Finally, we write
\begin{align*}
| \pr ({\cal T}_{1,2}) - \pr (\hat d_1 = s) \pr(\hat d_2 = t) | \wle &| \pr ({\cal T}_{1,2}) - \pr (\hat d_1 = s, \hat d_2 = t) | \\ &+ | \pr (\hat d_1 = s, \hat d_2 = t) -\pr (\hat d_1 = s) \pr(\hat d_2 = t) |.
\end{align*}
The first term on the right-hand side is at most $2\varepsilon$ by the above calculation, and the second term is $o(1)$ by Lemma \ref{the:truncation}.  Hence (\ref{2021-05-24}) follows.
  Claim \eqref{eq:minfinity} follows from (\ref{2021-05-24}), \eqref{eq:d1tohatd1}.  \qed
\end{proof}

We now introduce some notation. Recall that the bidegree distribution of $G_n$ is given by
\[
 f_{2,n}(s,t)
 \weq \pr\Big( \deg_{G_n}(1) = s, \deg_{G_n}(2) = t \ \big| \ \{1,2\}  \in E(G_n) \Big).
\]
For $A \subset [m]$, denote by $G_{n,A}$ the graph with $V(G_{n,A}) = [n]$ and $E(G_{n,A}) = \cup_{k \in A} E(G_{n,k})$. Denote $D_i =\deg_{G_n}(i)$. We note that for any $k$,
\[
 D_i
 \weq D_{i,k} + \tilde D_{i,k} - \hat D_{i,k},
\]
where $D_{i,k}$ and $\tilde D_{i,k}$ are the degrees produced by layer $k$ and by the layers other than $k$, i.e., 
\[
 D_{i,k} = \deg_{G_{n,k}}(i),
 \quad \tilde D_{i,k} = \deg_{G_{n, [m]\setminus \{k\}}}(i).
 \]
Furthermore, $\hat D_{i,k}$ stands for the number of neighbours produced by the layer $k$ and at least one other layer,
 \[
 \quad \hat D_{i,k} = \deg_{G_{n,k} \cap G_{n, [m]\setminus \{k\}}}(i).
\]
Denote the event $\cE_k = \{ \{1,2 \} \in E(G_{n,k})\}$ and 
\begin{align*}
 f_n(s) &\weq \pr(D_i = s), \\
 \tilde f_{2,n}(s,t)
 &\weq \pr( \tilde D_{1,k} =s, \tilde D_{2,k} = t), \\
 f'_{2,n}(s,t)
 &\weq \pr( D_{1,k} =s, D_{2,k} = t \cond \cE_k ).
\end{align*}
Throughout the proof of Theorem \ref{the:LimitingModeBidegree}:\eqref{ite:Weak} we write $X_i = X_{n,i}$ and $Y_i = Y_{n,i}$.

We summarize the proof strategy as follows. 
We show that the probability of the event
\[
\cE_k \cap (\{\hat D_{1,k} > 0 \} \cup \{ \hat D_{2,k} > 0\})
\]
is negligible (eq. \eqref{eq:on2}). Furthermore, for every $k=1,\dots, m$, on the event $\cE_k$, we have $D_1 \approx D_{1,k} + \tilde D_{1,k}$ and $D_2 \approx D_{2,k} + \tilde D_{2,k}$ for large $n$. 
This together with Theorem \ref{the:kovo20-3new} and Lemma \ref{the:minfinity} allows us to approximate $ f_{2,n} \approx (f_n \otimes f_n) \ast f'_{2,n} $ (eq. \eqref{eq:f2npapprox} and \eqref{eq:BidegreeKeyLimit}). In the last part of the proof we verify that $f'_{2,n}(s,t) \approx f_2'(s-1,t-1)$, which together with the previous approximation and $f_n \otimes f_n \approx \bar f_1 \otimes \bar f_1$ gives the result, $f_{2,n} \approx \delta_{(1,1)} \ast (\bar f_1 \otimes \bar f_1) \ast f_2'$. 

\begin{proof}[ Proof of Theorem~\ref{the:LimitingModeBidegree}:(i)]
First note that $\pr( \{1,2\} \in E(G_n)) = \pr(\cup_{k=1}^m \cE_k)$ and
\[
\sum_{k=1}^m \pr (\cE_k) - \sum_{k<l} \pr (\cE_k \cap \cE_l) \wle \pr (\cup_{k=1}^m \cE_k) \wle \sum_{k=1}^m \pr (\cE_k),
\]
and since the layers are i.i.d.,
\[
\sum_{k<l} \pr (\cE_k \cap \cE_l) = \frac{1}{2}(m)_2 (P_{21}^{(n)} (n)_2^{-1})^2,
\]
and so
\begin{equation}
\label{eq:cEk}
\pr( \{1,2 \} \in E(G_n)) = \pr (\cup_{k=1}^m \cE_k) = \sum_{k=1}^m \pr (\cE_k) + O(n^{-2}) = m\pr(\cE_1) + O(n^{-2}),
\end{equation}
where we used the fact that $m/n$ is bounded.
We similarly approximate
\begin{align}
&f_{2,n}(s,t) \, \pr( \{1,2\} \in E(G_n) )  \nonumber\\ 
 &\weq  \pr(D_1=s, D_2=t, \cup_k \cE_k ) \nonumber\\
 & \weq  m \pr( D_1=s, D_2=t, \cE_1 ) + O(n^{-2}) \nonumber\\
&\weq m \pr( D_{1,1} + \tilde D_{1,1} - \hat D_{1,1}=s, \, D_{2,1} + \tilde D_{2,1} - \hat D_{2,1}=t, \, \cE_1) + O(n^{-2}). \label{eq:hatD}
\end{align}

Next we show that on the event $\cE_1$ the probability that $\hat D_{1,1}>0$ or $\hat D_{2,1}>0$ is negligible. Recall that $X_1$ and $Y_1$ denote the size and strength of layer 1. The union bound gives
\begin{align*}
\pr(\cE_1 & \cap (\{\hat D_{1,1} > 0 \} \cup \{ \hat D_{2,1} > 0\}) \; | \; V(G_{n,1}), X_1, Y_1) \\
& \wle 2 \pr(\cE_1 \cap \{\hat D_{1,1} > 0 \} \; | \; V(G_{n,1}), X_1, Y_1)
\\
& \wle 2 \pr(\cE_1 \cap \left( \cup_{k=2}^m \cup_{j \in V(G_{n,1})} \{ \{1,j \} \in E(G_{n,k}) \} \right)  | \; V(G_{n,1}), X_1, Y_1) \\
& \wle 2 \sum_{k=2}^{m} \sum_{j \in V(G_{n,1})} \pr(\cE_1 \cap \{ \{1,j \} \in E(G_{n,k}) \}  \; | \; V(G_{n,1}), X_1, Y_1).
\end{align*}
Since the layers are i.i.d., this equals
\begin{align*}
2(m-1)&(X_1-1) \, \pr(\cE_1 \; | \; V(G_{n,1}), X_1, Y_1) \, \pr( \{1,j\} \in E(G_{n,2})   |  V(G_{n,1}), X_1, Y_1) \\
&\weq 2(m-1)(X_1-1) \, \Big( \I (\{ 1,2 \} \subset V(G_{n,1}) ) Y_1 \Big) \, \E \Big (\frac{(X_2)_2}{(n)_2} Y_2 \Big).
\end{align*}
Taking $\E[\, \ldots  | \, X_1,Y_1] $ of the above gives
\begin{align}
\pr(\cE_1 & \cap (\{\hat D_{1,1} > 0 \} \cup \{ \hat D_{2,1} > 0\}) \; | \;  X_1, Y_1)  \nonumber\\
&\wle  2(m-1)(X_1-1) Y_1 \frac{(X_1)_2 }{(n)_2} \E \Big (\frac{(X_2)_2}{(n)_2} Y_2 \Big). \label{eq:conditionX1Y1}
\end{align}
We now use the simple identity
\begin{align*}
(X_1)_2 Y_1 &\weq (X_1)_2 Y_1 \I ((X_1)_2 Y_1 < n^{1/2}) + (X_1)_2 Y_1 \I ((X_1)_2 Y_1 \geq n^{1/2}).
\end{align*}
Together with $X_1 \leq n$ this yields
\begin{align*}
(X_1)_2 Y_1 / n &\wle n^{-1/2} + X_1 Y_1 \, \I((X_1)_2 Y_1  \geq n^{1/2}),
\end{align*}
and it follows that
\begin{align*}
\E[ (X_1-1) Y_1 (X_1)_2 / n ] \wle \E[ (X_1-1) n^{-1/2} ]  + \E[(X_1)_2 Y_1 \, \I((X_1)_2 Y_1 \geq n^{1/2})].
\end{align*}
The first term goes to zero by $P_{10}^{(n)} \to P_{10}$. Since it was assumed that $P^{(n)}$ converges weakly and $P_{21}^{(n)} \to P_{21}$, it follows that $(X_1)_2 Y_1$ is uniformly integrable, so the second term also goes to zero. 
 Thus, $\E[ (X_1-1) Y_1 (X_1)_2 / (n)_2 ] = o(n^{-1})$, and so \eqref{eq:conditionX1Y1} gives
\begin{align}
\begin{aligned}
\label{eq:on2}
\pr(\cE_1 \cap (\{\hat D_{1,1} > 0 \} \cup \{ \hat D_{2,1} > 0\})) &\wle 2(m-1) \E \Big (\frac{(X_2)_2}{(n)_2} Y_2) \Big) o(n^{-1}) \\
&\weq o(n^{-2}).
\end{aligned}
\end{align}
Returning to \eqref{eq:hatD}, it follows that
\begin{align}
f&_{2,n}(s,t) \,  \pr( \{1,2\}  \in E(G_n) ) \nonumber \\
& \weq m \pr( D_{1,1} + \tilde D_{1,1} =s, \, D_{2,1} + \tilde D_{2,1} = t, \, \cE_1) +o(n^{-1}) \nonumber  \\
 &\weq m \sum_{s_1 \le s} \sum_{t_1 \le t} \tilde f_{2,n}(s_1,t_1)  \, \pr( D_{1,1} =s-s_1, D_{2,1} = t - t_1, \, \cE_1 ) + o(n^{-1}) \nonumber \\
 &\weq m \pr(\cE_1) \sum_{s_1 \le s} \sum_{t_1 \le t} \tilde f_{2,n}(s_1,t_1) f_{2,n}'(s-s_1,t-t_1) + o(n^{-1}) \nonumber \\
 & \weq  \pr( \{1,2\}  \in E(G_n) )
   \sum_{s_1 \le s} \sum_{t_1 \le t} \tilde f_{2,n}(s_1,t_1) \, f_{2,n}'(s-s_1,t-t_1) + o(n^{-1}), \label{eq:f2npapprox}
\end{align}
where we have used \eqref{eq:on2} and the fact that $m/n$ is bounded. In the last step we invoked \eqref{eq:cEk}.

We now apply Lemma \ref{the:minfinity} to approximate the term $\tilde f_{2,n}(s_1,t_1)$, which represents the joint degree distribution in the model with $m-1$ layers. First, observe that the conditions in Lemma \ref{the:minfinity}, $P^{(n)} \to P$ weakly and $P_{10}^{(n)} \to P_{10} \in (0,\infty)$, are satisfied by our assumptions. Secondly, since $m/n \to \mu$, also $(m-1)/n \to \mu$. Thus, we can apply \eqref{eq:minfinity} (with $\tilde f_{2,n}(s_1,t_1)$ in place of $\pr(D_1=s,D_2=t)$, and similarly  $\pr(\tilde D_{1,k} = s_1)$, $\pr(\tilde D_{2,k} = t_1)$ in place of $\pr(D_1=s)$, $\pr(D_2=t)$), and obtain
\[
\tilde f_{2,n}(s_1,t_1) - \pr( \tilde D_{1,k} =s_1) \pr( \tilde D_{2,k} = t_1) \to 0.
\]
Next, we approximate $\pr( \tilde D_{1,k} =s_1)$ and $\pr( \tilde D_{2,k} =t_1)$. Namely, by Theorem \ref{T1}, the asymptotic degree distribution only depends on $\mu$  and the limiting type distribution $P$, and in particular, replacing $m$ by $m-1$ does not change these limits. Hence, $\pr( \tilde D_{1,k} =s_1)$ and $f_n(s_1)$ converge to the same number,
\[
\pr( \tilde D_{1,k} = s_1) - f_n(s_1) \to 0, \quad \text{and} \quad \pr( \tilde D_{2,k} = t_1) - f_n(t_1) \to 0,
\]
which gives the approximation
\begin{align*}
&f_{2,n}(s,t)  \pr( \{1,2\} \in E(G_n)) \\
&= \pr( \{1,2 \}  \in E(G_n) )
   \sum_{s_1 \le s} \sum_{t_1 \le t} \Big( f_{n}(s_1) f_n(t_1) \, f_{2,n}'(s-s_1,t-t_1) + o(1) \Big) \!+\! o(n^{-1}),
\end{align*}
where we note that since $s$ and $t$ do not depend on $n$, the  double sum equals $(f_n \otimes f_n) \ast f'_{2,n}) (s,t) + o(1)$. 
Furthermore, the identity $\pr(\cE_1) = \frac{P_{21}^{(n)}}{(n)_2}$ together with \eqref{eq:cEk} gives
$\pr( \{1,2 \} \in E(G_n)) = \frac{P_{21}^{(n)}}{(n)_2} m+O(n^{-2}) = \operatorname{\Theta}(n^{-1})$.  
As a consequence, dividing by  $\pr( \{1,2 \} \in E(G_n))$ yields
\begin{equation}
 \label{eq:BidegreeKeyLimit}
 \abs{ f_{2,n}(s,t) - ((f_n \otimes f_n) \ast f'_{2,n}) (s,t) }
 \wto 0
\end{equation}
for any $s,t \in \Z_+$, with $\ast$ denoting the convolution of probability measures on the additive group $\Z^2$. We know that $f_n \to \bar f_1$ weakly where $\bar f_1$ is the limiting model degree distribution in \eqref{eq:LimitingModelDegree}. Therefore, $f_n \otimes f_n \to \bar f_1 \otimes \bar f_1 $ weakly as probability measures on $\Z_+^2$. 

Let us investigate the limit of $f'_{2,n}$. We note that given $(X_k,Y_k)$ and the event $\cE_k = \{ \{1,2\} \in E(G_{n,k})\}$, the random variables $D_{1,k}$ and $D_{2,k}$ are independent, and both distributed according to $1 + \Bin(X_k-2,Y_k)$. Hence
\begin{align*}
 &\pr( D_{1,k} =s, D_{2,k} = t, \cE_k \cond X_k, Y_k) \\
 &\weq \pr( \cE_k \cond X_k, Y_k) \, \pr( D_{1,k} =s, D_{2,k} = t \cond \cE_k, X_k, Y_k) \\
 &\weq \frac{(X_k)_2}{(n)_2} Y_k \, \Bin(X_k-2, Y_k)(s-1) \Bin(X_k-2,Y_k)(t-1).
\end{align*}
By taking expectations above, and dividing the outcome by $\pr(\cE_k) = \E \frac{(X_k)_2}{(n)_2} Y_k = (n)_2^{-1} P^{(n)}_{21}$, it follows that
\begin{align*}
 f'_{2,n}(s,t)
 &\weq \int \Bin(x-2,y)(s-1) \Bin(x-2,y)(t-1) \frac{(x)_2 y \, P^{(n)}(dx,dy)}{P^{(n)}_{21}} .
\end{align*}
When $P^{(n)} \to P$ weakly and $P^{(n)}_{21} \to P_{21} \in (0,\infty)$, it follows that $f'_{2,n}(s,t) \to f'_{2}(s-1,t-1)$ pointwise on $\Z_+^2$, where $f'_2$ is defined by \eqref{eq:BidegreePrime}.  Hence
\[
 (f_n \otimes f_n) \ast f'_{2,n}
 \wto \delta_{(1,1)} \ast (\bar f_1 \otimes \bar f_1)  \ast f'_{2}
\]
pointwise, where $\delta_{(1,1)}$ represents the edge between the two nodes. Combining this with \eqref{eq:BidegreeKeyLimit}, we conclude that Theorem~\ref{the:LimitingModeBidegree}:\eqref{ite:Weak} is valid.
\qed
\end{proof}

\subsection{Proof of Theorem~\ref{the:LimitingModeBidegree}:\eqref{ite:Wasserstein}}

The following lemma gives sufficient conditions for the convergence of the third moment of the model degree distribution. The proof of Theorem~\ref{the:LimitingModeBidegree}:\eqref{ite:Wasserstein} then follows from Skorohod's coupling theorem and basic properties of size-biased distributions. 
\begin{lemma}
\label{the:DegreeMomentConvergence}
Assume that $P^{(n)} \to P$ weakly and $P^{(n)}_{rs} \to P_{rs} < \infty$ for $rs = 10,21,32,43$, with $P_{10} > 0$.
Then the third moments of the model degree distribution converge according to $\sum_s s^3 f_n(s) \to \sum_s s^3 {\bar f}(s) < \infty$. 
\end{lemma}
\begin{proof}
Let $D_n = \deg_{G_n}(1)$ and let $D_*$ be a random variable with the asymptotic degree distribution defined by Theorem \ref{T1}. We write for short $X_i = X_{n,i}$ and $Y_i = Y_{n,i}$. Since $D_n \rightarrow D_*$ weakly, by portmanteau theorem and Fatou's lemma
\begin{align*}
\E (D_*^3) = \sum_{k=0}^\infty \pr (D_* \!>\! k^{1/3}) \leq \sum_{k=0}^\infty \liminf_n \pr (D_n \!>\! k^{1/3}) 
&\leq \liminf_n \sum_{k=0}^\infty \pr(D_n \!>\! k^{1/3}) \\
&= \liminf_n \E (D_n^3).
\end{align*}
It remains to show that $\limsup_n \E (D_n^3) \leq \E (D_*^3)$. Let $z_k \!=\! \sum_{j=2}^n \I ( \{1,j\} \! \in \! E(G_{n,k})) $  be the number of neighbours of $1$ produced by the layer $k$. Then $z_k$, $1 \leq k \leq m$, are independent and $D_n \leq \sum_{k=1}^m z_k $, so that
\begin{align*}
\E (D_n^3) \wle \E \Big( (\sum_{k=1}^m z_k)^3 \Big) &\weq m \E (z_1^3) + 6 \binom{m}{2} \E (z_1^2 z_2)+ 6\binom{m}{3} \E (z_1 z_2 z_3) \\
&\wle m \E (z_1^3) + 3(m \E(z_1^2))(m \E(z_1)) + (m \E(z_1))^3.
\end{align*}
Recall that the first three moments of $\Bin(n,p)$ are
\begin{gather}
\label{eq:binommoments}
\mu_1 = np, \quad \mu_2 = np + (n)_2 p^2, \quad \mu_3 = np + 3 (n)_2 p^2 + (n)_3 p^3.
\end{gather}
Introduce the event $A = \{1 \in V(G_{n,1}) \}$. On the event $A^c$ we have $z_1 = 0$. Given $(X_k, Y_k)$ and $A$, the random variable $z_k$ is $\Bin(X_k-1,Y_k)$-distributed. Hence by \eqref{eq:binommoments}
\begin{align*}
\E (z_1) \weq \E[ \E( \, \I_A z_1 \; | \; X_1, Y_1 \, ) ] &\weq \E [\pr(A \;| \;X_1, Y_1) \E (z_1 \;|\; X_1, Y_1, A) ] \\
&\weq \E \Big( \frac{X_1}{n} (X_1-1) Y_1 \Big) \weq \frac{1}{n} P^{(n)}_{21},
\end{align*}
and similarly
\begin{align*}
\E (z_1^2) &\weq \E \Big( \frac{X_1}{n} \Big( (X_1-1) Y_1 + (X_1-1)_2 Y_1^2 \Big) \Big) \weq \frac{1}{n}  (P^{(n)}_{21} +  P^{(n)}_{32}), \\
\E (z_1^3) &\weq \E \Big( \frac{X_1}{n} \Big( (X_1-1) Y_1 + 3(X_1-1)_2 Y_1^2 + (X_1-1)_3 Y_1^3 \Big) \Big) \\
&\weq \frac{1}{n} (P^{(n)}_{21} + 3  P^{(n)}_{32} +  P^{(n)}_{43})
\end{align*}
Since $P^{(n)}_{rs}$ was assumed to converge for $rs = 10, 21, 32, 43$, it follows that 
\begin{align}
\label{eq:Dthirdmoment}
\limsup_n \E D_n^3 \wle  \mu (P_{21} + 3  P_{32} +  P_{43}) + 3 \mu^2(P_{21} +  P_{32}) P_{21} + \mu^3  P_{21}^3. 
\end{align}
Recall that $D_*$ may be represented as $D_* = \sum_{j=1}^\Lambda H_j$ where $\Lambda \sim \Poi(\mu P_{10})$ and
\[
\pr(H=l) \weq \frac{1}{P_{10}} \E \left(X \binom{X-1}{l}Y^l(1-Y)^{X-1-l} \right),
\quad l=0,1,2,\dots,
\]
where $\law(X,Y) = P$. From this and \eqref{eq:binommoments} we obtain
\begin{align*}
\E (H) \weq \sum_{l=0}^\infty l \,\pr(H=l) &\weq \frac{1}{P_{10}} \E \left(X \sum_{l=0}^\infty l \binom{X-1}{l}Y^l(1-Y)^{X-1-l} \right) \\
&\weq \frac{1}{P_{10}} \E \left(X (X-1) Y \right) \weq \frac{P_{21}}{P_{10}},
\end{align*}
and similarly
\begin{align*}
\E (H^2) &\weq \frac{1}{P_{10}} \E [ X ((X-1) Y + (X-1)_2 Y^2) ] \weq \frac{P_{21} + P_{32}}{P_{10}}, \\
\E (H^3) &\weq \frac{1}{P_{10}} \E [ X ((X-1) Y + 3(X-1)_2 Y^2 + (X-1)_3 Y^3) ] \\ &\weq \frac{P_{21} + 3P_{32} + P_{43}}{P_{10}}.
\end{align*}
Recall that the factorial moments of $\Poi(\lambda)$ are given by $\lambda^k$. It follows that
\begin{align*}
\E (D_*^3) \weq \E[\E(D_*^3 \, | \, \Lambda )] &\weq \E[\Lambda \E (H_1^3) + 3(\Lambda)_2 \E (H_1^2)\E (H_1) + (\Lambda)_3 \E (H_1)^3] \\
&\weq \mu (P_{21} + 3P_{32} + P_{43}) + 3 \mu^2 (P_{21} + P_{32})P_{21} + \mu^3 P_{21}^3 .
\end{align*}
The claim now follows by \eqref{eq:Dthirdmoment}.
\qed
\end{proof}

We are now ready to prove Theorem~\ref{the:LimitingModeBidegree}:\eqref{ite:Wasserstein}. The proof is similar to \cite[Theorem 3.2]{VanDerHofstad_Litvak_2014}, but slightly simpler because here we analyze model distributions instead of empirical distributions of random graph samples.

\begin{proof}[Proof of Theorem~\ref{the:LimitingModeBidegree}:(ii)]
Let $(D_{1,n}^*, D_{2,n}^*)$ be a random variable distributed according to the model bidegree distribution $f_{2,n}$ of $G_n$. Theorem~\ref{the:LimitingModeBidegree}:\eqref{ite:Weak} states that $(D_{1,n}^*, D_{2,n}^*) \!\to\! (D_{1}^*, D_{2}^*)$ weakly. Now let $\phi: \Z_+^2 \to \R$ be a function bounded by $\abs{\phi(x,y)} \le c(1 + x^2 + y^2)$. Skorohod's coupling theorem \cite[Theorem 4.30]{Kallenberg_2002} implies that there exist a probability space and some random variables $(\tilde D_{1,n}^*, \tilde D_{2,n}^*)  \stackrel{d}{=}   (D_{1,n}^*, D_{2,n}^*)$ and $(\tilde D_{1}^*, \tilde D_{2}^*) \stackrel{d}{=} (D_{1}^*, D_{2}^*)$ such that $(\tilde D_{1,n}^*, \tilde D_{2,n}^*) \to (\tilde D_{1}^*, \tilde D_{2}^*)$ almost surely. Then $Z_n := \phi(\tilde D_{1,n}^*, \tilde D_{2,n}^*) \to \phi(\tilde D_{1}^*, \tilde D_{2}^*) =: Z$  almost surely. Also $\abs{Z_n} \le c( 1+  (\tilde D_{1,n}^*)^2 + (\tilde D_{2,n}^*)^2) =: Z_n'$ a.s. 

Denote by $f_n$ the model degree distribution of $G_n$, and by $f_n^*$ its size-biased version. Recall that $\law(\tilde D_{1,n}^*) = \law(D_{1,n}^*)$ equals the first (equivalently, the second) marginal of $f_{2,n}$, and as in \eqref{eq:sizebias2}, this marginal equals
\[
 \sum_t f_{2,n}(s,t) \weq f_n^*(s)
 \weq \frac{s f_n(s)}{\sum_t t f_n(t)},
\]
hence $\law(D_{1,n}^*) = f^*_n$. We know (see for example \cite[Lemma 1.23]{Kallenberg_2002}) that for the size-biasing $f^*_n$, and for any measurable nonnegative function $\phi$,
\[
 \E \phi(D_{1,n}^*)
 \weq \sum_s \phi(s) f^*_n(s)
 \weq \frac{\sum_s \phi(s) s f_n(s)}{\sum_s s f_n(s)}
 \weq \frac{\E \phi(D_{1,n}) D_{1,n}}{\E D_{1,n}},
\]
where $\law(D_{1,n}) = f_n$. Especially, for $\phi(s) = s^2$, it follows that
\[
 \E ((D_{1,n}^*)^2)
 \weq \frac{\E D_{1,n}^3}{\E D_{1,n}}.
\]

With the help of Lemma~\ref{the:DegreeMomentConvergence}, we now note that
\[
 \E ((\tilde D_{1,n}^*)^2)
 \weq \frac{\E D_{1,n}^3}{\E D_{1,n}}
 \wto \frac{\E D_{1}^3}{\E D_{1}}
 \weq \E ((D_{1}^*)^2) < \infty,
\]
and hence $\E Z_n' \to \E Z' = c(1+2 \E ((D_{1}^*)^2)) < \infty$. Lebesgue's dominated convergence theorem (see the version in \cite[Theorem 1.21]{Kallenberg_2002}) now implies that $\E Z_n \to \E Z$, which confirms the claim.
\qed
\end{proof}
\subsection{Proof of Theorem \ref{the:LL1}}
We use the two following results shown in \cite{Bloznelis_Leskela_2019-12}. Theorem \ref{the:R1new} and Lemma \ref{the:kovo20-3new} give the power laws of $D_1$ and $D_1'$ when the assumptions of Theorem \ref{the:LL1} hold. The first claim \eqref{2020-02-24++5} then follows from the observation that $D_1^*$ has the same distribution as $1+D_1+D_1'$, and that $D_1'$ is the dominating term. In the proof of (\ref{2020-02-24++9}) we use the fact that random variables $D_1$, $D_2$ and $(D_1',D_2')$ are   independent and exploit the special structure of $(D_1',D_2')$: it is a mixture of (conditionally) independent binomial random variables.

\begin{theorem}[\cite{Bloznelis_Leskela_2019-12}, Theorem 4.1] \label{the:R1new}
Assume that the limiting layer type distribution equals $P(dx,dy) = p(dx)\delta_{q(x)} (dy)$ with
\begin{gather*}
p(x) \weq (a+o(1))x^{-\alpha} \quad \text{and} \quad q(x) \weq (b+O(x^{-1/2})) x^{-\beta}
\end{gather*}
as $x \to + \infty$ with exponents $\alpha>2$, $\beta \in (0,1)$ and constants $a,b > 0$. Then the limiting degree distribution satisfies
\begin{equation}
\label{2021-04-24****5}
\pr (D_1 =t ) \weq dt^{-\delta}(1+o(1)), \quad \text{as $t \to +\infty$,}
\end{equation}
where $\delta = 1+ \frac{\alpha-2}{1-\beta}$ and $d=\mu(1- \beta)^{-1} a b^{\delta-1}$. The same result holds for $\beta = 0$ if $b<1$.
\end{theorem}

\begin{lemma}[\cite{Bloznelis_Leskela_2019-12}, Lemma A.4] \label{the:kovo20-3new}
Consider a mixed binomial distribution $g(r) = \sum_{k\geq 1} p_k f_k(r)$ where $f_k = \Bin(x_k,y_k)$ and $(p_k)$ is a probability distribution on $\{1,2,\ldots \}$. Assume that as $k \to +\infty$
\begin{gather*}
x_k = (a+O(k^{-\xi/2}))k^{\xi}, \quad y_k = (b+O(k^{-\xi/2}))k^{-\eta}, \quad p_k = (c+o(1))k^{-\gamma},
\end{gather*}
for some $0 \leq \eta < \xi < \eta+2$ and $\gamma > 1$, and some $a,b,c > 0$ such that $\eta>0$ or $b<1$. Then
\[
g(r) \weq (d+o(1))r^{-\delta} \qquad \text{as $r \to +\infty$,}
\]
where $\delta = 1+ \frac{\gamma-1}{\xi-\eta}$ and $d= (ab)^{\delta-1} c/(\xi-\eta)$.
\end{lemma} 

\begin{proof}[Proof of Theorem \ref{the:LL1}]
We first prove claim \eqref{2020-02-24++5}, $\pr(D_1^* = t) \weq (1+o(1))c' t^{-\frac{\alpha-2}{1-\beta}}$. Recall from Theorem \ref{the:LimitingModeBidegree} that the asymptotic bidegree distribution can be written as
\begin{equation}
\label{2021-04-24****1}
\pr(D_1^* = s+1 , D_2^* = t +1) \weq \pr (D_1 + D_1' = s, D_2+D_2' = t),
\end{equation}
where $D_1$ and $D_2$ follow the distribution $\bar f_1$ in  \eqref{eq:LimitingModelDegree} and $(D_1', D_2')$ follows the distribution $f_2'$ in  \eqref{eq:BidegreePrime}, $(D_1', D_2')$ is independent of $(D_1, D_2)$, and $D_1$ is independent of $D_2$.
Note that $D_1'$ is a mixed binomial random variable $D_1' \sim \Bin(X'-2, Y') $, where $(X',Y')$ has the distribution 
\[
\pr(X'=t, Y'=q(t)) \weq \pr(X'=t) \weq \pr(X=t) \frac{(t)_2q(t)}{\E((X)_2 q(X))}, \quad t=0,1,2,\ldots
\]
In Lemma \ref{the:kovo20-3new}, set $x_k = k-2$, $y_k = q(k) = bk^{-\beta}$, and $p_k = (k)_2 q(k) p(k) / P_{21} $, where we note that the exponent of $p_k=O(k^{2-\alpha-\beta})$ is less than $-1$ by the assumption $\alpha+\beta>3$. Then
\begin{align}
\label{eq:D1primeasymp}
  \pr(D_1' = t) \weq (c' + o(1))t^{-\frac{\alpha-2}{1-\beta}}.
\end{align}
Denote the power law exponents of $D_1$ (given by Theorem \ref{the:R1new}) and $D_1'$ by
\[
\alpha_{D} \,:=\, 1+(\alpha-2)/(1-\beta), \qquad \alpha_{D'} \,:=\, (\alpha-2)/(1-\beta).
\]
A standard argument (\cite{mikosch1999}) shows that  $\pr(D_1+D_1' > t) \sim \pr(D_1 > t) + \pr(D_1' > t)$.  Now  $\alpha_{D} > \alpha_{D'}$ implies
\[
\pr (D_1 + D_1' = t) = (1+o(1)) \pr(D_1' = t).
\]
From (\ref{2021-04-24****1}) and  (\ref{eq:D1primeasymp})
we obtain $ \pr(D_1^*=t)=(1+o(1))c' t^{-\frac{\alpha-2}{1-\beta}}$, thus showing  \eqref{2020-02-24++5}.
Let us prove 
\eqref{2020-02-24++9}, i.e., $\pr ( D_1^* = t_1, D_2^* = t_2 ) = (1+o(1))c''(t_2-t_1)^{-1-\frac{\alpha-2}{1-\beta}} t_1^{-\frac{\alpha-2}{1-\beta}}$.
We start with an outline of the proof.
In order to determine the asymptotics of the bivariate probability (\ref{2021-04-24****2}) we use the observation 
that large values of mixed binomial random vector $(D_1',D_2')$ concentrate around  the diagonal $D_1'=D_2'$. We justify this claim by proving that values 
laying far apart from the diagonal have superpolynomially small probabilities. Namely, for  any $s_0>0$ and all
$s_0<s+\delta_s+\delta_t<t$ we have 
\begin{equation}
\label{2020-02-27}
\pr(D_1'=s,D_2'=t) \le e^{-c_{**}\ln^8t},
\end{equation}
where the constant $c_{**}>0$ depends on $s_0$, 
$\beta$, and $b$. The particular form of $\delta_t=t^{1/2}\ln^4(2+t)$ is related to the  exponential bounds for binomial probabilities below. With this observation in mind we expand the probability
\begin{equation}
\label{2021-04-24****2}
\pr(D_1^*=t_1+1,D_2^*=t_2+1) = \pr(D_1+D_1'=t_1,D_2+D_2'=t_2)
\end{equation}
into the sum  
\begin{equation}\label{2021-04-24**15}
\sum_{(i,j)\in[0,t_1]\times [0,t_2]}\pr (D_1'=i,D_2'=j)
\pr(D_1=t_1-i)
\pr(D_2=t_2-j),
\end{equation}
locate a region of $(i,j)$ that gives the leading term (see region $A_{4}$ below) and show that the contribution of the remaining part  is negligible.
Note that our assumptions $t_2>t_1$ and $\delta_{t_2}=o(t_2-t_1)$ make the sum (\ref {2021-04-24**15}) asymmetric.

Let us prove (\ref{2020-02-27}).
In the proof we use the upper bound for binomial probabilities
\begin{equation}
\label{2021-04-24**1}
\Bin(x,p)(s) \wle e^{-0.5\frac{\delta^2}{(s+\delta)}}, \qquad \forall s,\delta: \, \delta>0, \, \abs{s-xp}  \geq \delta
 \end{equation}
that follows from Chernoff bounds (see, e.g., Theorem 2.1 in \cite{Janson_Luczak_Rucinski_2000}) combined with
the simple inequality $\pr(X = t) \leq \min \{\pr(X \leq t), \pr (X \geq t) \}$ (cf. \cite{Bloznelis_Leskela_2019-12}).

Setting $p=q(x)$ and $\delta=\delta_s$  in 
(\ref{2021-04-24**1}) gives
\begin{align*}
\Bin(x,q(x))(s) \wle e^{-0.5\frac{\delta_s^2}{s+\delta_s} } \weq e^{-0.5 \frac{s\ln^8(2+s)}{s+s^{1/2}\ln^4(2+s)} },
\end{align*}
and approximating the exponent for large $s$ gives
\begin{align}
\label{eq:chernoff}
& \max_{x:\,|s-xq(x)|\ge \delta_s}\Bin(x,q(x))(s) \wle (1+o(1))e^{-0.5\ln^8s}.
\end{align}
 From (\ref{2021-04-24**1})
 we also obtain that 
\begin{equation}
\label{2021-04-24**2}
\Bin(x,q(x))(s)\le e^{-c_*\ln^8(xq(x))} \quad {\text{for}} \quad  xq(x) > s+\delta_s,
\end{equation}
where $c_*>0$ depends on $\beta$ and $b$. 
To show (\ref{2021-04-24**2}) we write $xq(x)$ in the form $xq(x)=s(1+\tau)$ and give a lower bound for the ratio
$(xq(x)-s)^2/(xq(x))=s\tau^2/(1+\tau)$ in the exponent of (\ref{2021-04-24**1}). For 
$s^{-0.5}\ln^4s\le \tau\le 1$ this ratio is at least
\begin{displaymath}
\frac{s\tau^2}{2} \ge \frac{\ln^8s}{2} \ge \frac{1}{2}\ln^8\left(\frac{xq(x)}{2}\right) \ge c_*\ln^8(xq(x)).
\end{displaymath}
For $\tau\ge 1$ the ratio is 
\begin{displaymath}
\frac{xq(x)}{(1+\tau)^2}\tau^2 \ge  \frac{xq(x)}{4} \ge c_*\ln^8(xq(x)).
\end{displaymath}

Now we are ready to prove (\ref{2020-02-27}).
Since $D_1', D_2'$ are conditionally independent and binomial, we have
\begin{displaymath}
\pr(D_1'=s, D_2'=t )  \weq  \sum_{x\ge s} \Bin(x-2,q(x))(s) \, \Bin(x-2,q(x))(t) \, \pr(X'=x).
\end{displaymath}
Furthermore, from (\ref{eq:chernoff}) and (\ref{2021-04-24**2}) we have
\begin{displaymath}
 \Bin(x-2,q(x))(s) \cdot \Bin(x-2,q(x))(t)  \le e^{-c_{**}\ln^8t}.
\end{displaymath}
Indeed,  for $xq(x)<t-\delta_t$ the second probability is at most of order $e^{-c_{**}\ln^8t}$ by (\ref{eq:chernoff}). For $xq(x)>t-\delta_t$ we have  $xq(x)>s+\delta_s$ and $xq(x)\ge t/2$. Now we use (\ref{2021-04-24**2}) to bound the first binomial probability  by $e^{-c_*\ln^8(t/2)}$ from above. The proof of (\ref{2020-02-27}) is complete.
 
Now we evaluate \eqref{2021-04-24**15}. Given $0<\varepsilon<0.1$, let 
\begin{displaymath}
t_1^*=\varepsilon \min\{t_1, t_2-t_1\}, \qquad
t_2^*=
\begin{cases}
(t_2/\varepsilon^{1+\alpha_D})^{1/\alpha_D} \quad {\text{for}} \quad  2t_1>t_2, \\
\varepsilon t_2, \qquad \qquad \qquad
{\text{for}} \quad  2t_1<t_2.
\end{cases}
\end{displaymath}
We split $[0,t_1]\times [0,t_2]=A_0\cup A_1\cup A_2\cup A_3$,
where 
\begin{align*}
A_0 &= [0, t_1/2]\times [0,t_1/2], \qquad \quad A_1 = \bigl( [0, t_1- t_1^*]\times[0,t_2-t_2^*] \bigr) \setminus A_0, \quad \\
A_2 &= (t_1-t_1^*, t_1]\times[0,t_2], \qquad \quad \! \! \! 
A_3 = [0,t_1-t_1^*]\times (t_2-t_2^*, t_2],
\end{align*} 
and split 
$A_2=A_{4}\cup A_{5}\cup A_{6}$, where
\begin{align*}
&A_{4} = \{(i,j): t_1-t_1^*\le i\le t_1, \,  i-3\delta_i\le j\le i+3\delta_{t_2}\},
\\
&A_{5} = \{ (i,j): t_1-t_1^*\le i\le t_1, \,  0\le j< i-3\delta_i \}, \\
&A_{6} = \{ (i,j): t_1-t_1^*\le i\le t_1, \,  i+3\delta_{t_2}< j\le t_2\}.
\end{align*}
We denote 
${\cal R}=t_1^{-\alpha_{D'}}(t_2-t_1)^{-\alpha_D}$ 
and write for short
\begin{gather*}
S_{A_k}=\sum_{(i,j)\in A_k}h(i,j),\\
 h(i,j)= \pr (D_1'=i,D_2'=j) \pr (D_1=t_1-i) \pr(D_2=t_2-j).
\end{gather*}
In order to determine the asymptotics of  (\ref{2021-04-24**15}) we split 
\begin{displaymath}
\pr(D_1+D_1'=t_1,\, D_2+D_2'=t_2) = S_{A_0}+S_{A_1}+S_{A_2}+S_{A_3},
\end{displaymath}
and show that
\begin{eqnarray}
\label{2021-04-24++10}
&&
S_{A_2}=(1+o_{\varepsilon}(1))\pr(D_2=t_2-t_1)\pr(D_1'=t_1)+o({\cal R}),
\\
\label{2021-04-24+1}
&&
S_{A_0}+S_{A_1}+S_{A_3}\le c\varepsilon {\cal R}+c\varepsilon^{-2\alpha_D}t_1^{1-\alpha_D}{\cal R}+ o({\cal R}).
\end{eqnarray}
Here $o_{\varepsilon}(1)\to 0$ as $\varepsilon \to 0$ and   $t_1,t_2\to+\infty$ so that $\delta_{t_2}=o(t_2-t_1)$. 
Note that (\ref{2021-04-24****2}), (\ref{2021-04-24++10}), (\ref{2021-04-24+1}) imply 
\begin{displaymath}
\pr(D_1^*=t_1+1,D_2^*=t_2+1) = (1+o(1))\pr(D_2=t_2-t_1)\pr(D_1'=t_1).
\end{displaymath}
Invoking the asymptotic formulae
(\ref{2021-04-24****5}), (\ref{eq:D1primeasymp}) for probabilities
$\pr(D_2=t_2-t_1)$ and $\pr(D_1'=t_1)$ as   $t_2-t_1\to+\infty$ and $t_1\to+\infty$ we obtain (\ref{2020-02-24++9}).

We complete the proof by showing (\ref{2021-04-24++10}) and (\ref{2021-04-24+1}). We first prove (\ref{2021-04-24+1}).
Using the inequalities  (see \eqref{2020-02-24++5}, (\ref{2021-04-24****5})) 
\begin{align*}
&\pr(D_1=t_1-i) \le c t_1^{-\alpha_D}, \quad \quad \,
\pr(D_2=t_2-j) \le ct_2^{-\alpha_D} \qquad \quad \; \text{for} \quad (i,j)\in A_0, \\
&\pr(D_1=t_1-i)\le c (t_1^*)^{-\alpha_D}, \quad \pr(D_2=t_2-j)\le c(t_2^*)^{-\alpha_D} \qquad {\text{for}} \quad (i,j)\in A_1, \\
&\quad \sum_{(i,j)\in A_1}\pr(D_1'=i,D_2'=j) \le \pr(D_1'\ge t_1/2)+\pr(D_2'\ge t_1/2) \le ct_1^{1-\alpha_{D'}},
\end{align*}
we bound the sums
\begin{eqnarray}
\nonumber
&&
S_{A_0} \le c t_1^{-\alpha_D} t_2^{-\alpha_D}  \sum_{(i,j)\in A_0} \pr(D_1'=i, D_2'=j)  \le  ct_1^{-\alpha_D} t_2^{-\alpha_D} = o({\cal R}), \\
\nonumber && S_{A_1} \le  c (t_1^*)^{-\alpha_D} (t_2^*)^{-\alpha_D} \sum_{(i,j)\in A_1} \pr(D_1'=i, D_2'=j) \\
\label{2021-04-23+2}
&&
\qquad \le c(t_1^*)^{-\alpha_D} (t_2^*)^{-\alpha_D} t_1^{1-\alpha_{D'}}.\end{eqnarray}
Note that the quantity in  (\ref{2021-04-23+2})
equals
$c\varepsilon (t_2-t_1)^{-\alpha_D}t_2^{-1}t_1^{1-\alpha_{D'}}\le c\varepsilon {\cal R}$  for $2t_1>t_2$. For $2t_1\le t_2$ this quantity  equals
$c\varepsilon^{-2\alpha_D} t_2^{-\alpha_D} t_1^{1-\alpha_D-\alpha_{D'}} \le  c\varepsilon^{-2\alpha_D}t_1^{1-\alpha_D}{\cal R}.$

Next we estimate $S_{A_3}$. We have
\begin{eqnarray}
\label{2021-04-22}
&& S_{A_3} \le  \sum_{(i,j)\in A_3} \pr(D_1'=i, D_2'=j) \le  \sum_{(i,j)\in A_3} e^{-c_{**}\ln^8{j}} \\
\label{2021-04-23+5}
&& \qquad \le  t_1t_2 e^{-c_{**}\ln^8(t_2-t_2^*)} \le e^{-c_{**}'\ln^8t_2} = o({\cal R}).
\end{eqnarray}
The first inequality of  (\ref{2021-04-22}) is obvious. The second one follows from
(\ref{2020-02-27}). To verify the condition
$i+\delta_i+\delta_j<j$ for $(i,j)\in A_3$ we write
\begin{displaymath}
i+\delta_i+\delta_j \le t_1 + \delta_{t_1} + \delta_{t_2} < t_2-t_2^* \le j.
\end{displaymath}
Note that the second inequality above follows from the fact that $\delta_{t_2}=o(t_2-t_1)$. In particular, this inequality is obvious for $2t_1<t_2$. For $2t_1>t_2$ the inequality follows from 
$t_2^*<\delta_{t_2}$ (note that $\alpha_{D}>2$).

\medskip

We secondly prove (\ref{2021-04-24++10}).
We split
$ S_{A_2} = S_{A_4}+ S_{A_5}+S_{A_6} $
and  show that $S_{A_4}$ satisfies 
(\ref{2021-04-24++10}) while 
$S_{A_5}+S_{A_6}=o({\cal R})$.
We derive the latter bound using
(\ref{2020-02-27}), which implies
\begin{align*}
&\pr(D_1'=i, D_2'=j) \le ce^{-c_{**}\ln^8i} \le  e^{-c_{**}'\ln^8 t_1} \quad
{\text{ for}} \quad (i,j)\in A_{5}, \\
&\pr(D_1'=i, D_2'=j) \le  ce^{-c_{**}\ln^8j} \le  e^{-c_{**}'\ln^8t_2} \quad
{\text{ for}}
\quad (i,j)\in A_{6}.
\end{align*}
In the very last step we invoked  inequalities $\ln j \ge  \ln (3\delta_{t_2}) \ge  0.5\ln t_2$. For $(i,j)\in A_{5}$ we will  use, in addition,  the inequality $\pr(D_2=t_2-j)\le c(t_2-t_1)^{-\alpha_D}$.
We have
\begin{displaymath}
S_{A_5} \le  c(t_2-t_1)^{-\alpha_D} \sum_{(i,j)\in A_5} \pr(D_1'=i,D_2'=j) \le c
(t_2-t_1)^{-\alpha_D} t_1^2 e^{-c'_{**}\ln^8t_1}.
\end{displaymath}
Here $t_1^2$ bounds the number of summands in the double sum. 
We conclude that  $S_{A_5}\le c (t_2-t_1)^{-\alpha_D}e^{-c''_{**}\ln^8t_1}=o({\cal R})$. The proof of $S_{A_6}=o({\cal R})$ is simpler,
\begin{equation}
\label{2021-04-23+11}
S_{A_6} \le \sum_{(i,j)\in A_6} \pr(D_1'=i,D_2'=j) \le t_1t_2e^{-c'_{**}\ln^8t_2} \le e^{-c''_{**}\ln^8t_2} = o({\cal R}),
\end{equation}
where $ t_1t_2$ bounds the number of summands in the double sum.

Now consider $S_{A_4}$. Denote  $h'_{i,j}= \pr(D_1'=i,D_2'=j) \pr(D_1=t_1-i)$ and put
\begin{align*}
S'_{A_4} = \sum_{(i,j)\in A_4} h'_{i,j}, \qquad S'_{A_2} = \sum_{(i,j)\in A_2} h'_{i,j}, \qquad S^*=\sum_{t_1-t_1^*\le i\le t_1}\sum_{j\ge 0}h'_{i,j}.
\end{align*}
Note that uniformly in $(i,j)\in A_4$ we have
\begin{equation}
\label{2021-04-23+12}
\pr(D_2=t_2-j) = (1+o_{\varepsilon}(1))\pr(D_2=t_2-t_1), 
\end{equation}
where $o_{\varepsilon}(1)\to 0$ as $\varepsilon \to 0$ and $t_1,t_2\to+\infty$, $(t_2-t_1)/\delta_{t_2}\to+\infty$. Furthermore, uniformly in $t_1-t_1^*\le i\le t_1$ we have as $t_1\to+\infty$
\begin{equation}
\label{2021-04-23+13}
\pr(D_1'=i) = (1+o_{\varepsilon}(1))\pr(D_1'=t_1).
\end{equation}
Now,
using (\ref{2021-04-23+12}) we write
\begin{equation}\label{2021-04-24**17}
S_{A_4}= (1+o_{\varepsilon}(1))\pr(D_2=t_2-t_1) S'_{A_4}.
\end{equation}
Then, proceeding as in the proof of $S_{A_5}+S_{A_6}=o({\cal R})$ above, we approximate
\begin{equation}\label{2021-04-24**16}
S'_{A_4}=S'_{A_2} + O\bigl(e^{-c'_{**}\ln^8t_1}\bigr) + O\bigl(e^{-c''_{**}\ln^8t_2}\bigr).
\end{equation}
Finally, we evaluate $S'_{A_2}$. We write $S'_{A_2}$ in the form $ S'_{A_2} = S^*-R^*$, where
\begin{align*}
R^* &= \sum_{t_1-t_1^*\le i\le t_1} \sum_{j> t_2} \pr(D_1'=i, D_2'=j) \pr(D_1=t_1-i) \\
&\le \sum_{t_1-t_1^*\le i\le t_1} \sum_{j> t_2}e^{-c_{**}\ln^8j} \\ 
&\le t_1e^{-c_{**}''\ln^8t_2} = o({\cal R}),
\end{align*}
where in the first inequality we use (\ref{2020-02-27}) and $\pr(D_1=t_1-i)\le 1$. We apply  (\ref{2021-04-23+13})  to the sum $S^*$,
\begin{align*}
S^* &= \sum_{t_1-t_1^*\le i\le t_1} \pr(D_1'=i) \pr(D_1=t_1-i) \\
&= (1+o_{\varepsilon}(1))\pr(D_1'=t_1) \sum_{t_1-t_1^*\le i\le t_1} \pr(D_1=t_1-i) \\
&= (1+o_{\varepsilon}(1))\pr(D_1'=t_1).
\end{align*}
We conclude that 
$S'_{A_2} = (1+o_{\varepsilon}(1))\pr(D_1'=t_1)+o({\cal R})$.
Now (\ref{2021-04-24**17}), (\ref{2021-04-24**16}) imply  
\begin{displaymath}
S_{A_4}=(1+o_{\varepsilon}(1)) \pr(D_2=t_2-t_1) \pr(D_1'=t_1)+o({\cal R}).
\end{displaymath}
The proof of (\ref{2021-04-24++10}) is complete.

\qed
\end{proof}

\subsection{Correlation of the limiting bidegree distribution}

Let us analyze the Pearson correlation coefficient $\Cor(D_1^*, D_2^*)$ of the limiting bidegree distribution in Theorem~\ref{the:LimitingModeBidegree}. 

\begin{proposition}
\label{the:LimitingBidegreeCorrelation}
For any $\mu \in (0,\infty)$ and any probability measure $P$ on $\Z_+ \times [0,1]$ such that $0 < P_{10}, P_{21} < \infty$ and $P_{32}, P_{43} < \infty$, the random variables $(D_1^*, D_2^*)$ in \eqref{eq:BidegreeRepresentation} satisfy
\[
 \Cor(D_1^*, D_2^*)
 \weq \frac{P_{21}( P_{43} + P_{33}) - P_{32}^2}
 {P_{21}(P_{43} + P_{32}) - P_{32}^2 + \mu P_{21}^2 ( P_{21} + P_{32})}.
\]
\end{proposition}
\begin{proof}
Recall the distribution of $(D_1', D_2')$,
\begin{equation*}
  f'_{2}(s,t) \weq \int_{\Z_+ \times [0,1]} \Bin(x-2,y)(s) \Bin(x-2,y)(t) \frac{(x)_2 y \, P(dx,dy)}{P_{21}}.
\end{equation*}
If $B$ is a $\Bin(x-2, y)$-distributed random variable, then $\E B = (x-2)y$ and $\E (B)_2 = (x-2)_2 y^2$, from which we conclude that $\E B^2 = \E (B)_2 + \E B = (x-2)_2 y^2 + (x-2)y$. Because $(x-2)(x)_2 = (x)_3$, it follows that
\[
 \E D'_1
 \weq \int (x-2)y \frac{(x)_2 y P(dx,dy)}{P_{21}}
 \weq \frac{P_{32}}{P_{21}}.
\]
Further, by noting that $(x-2)_2 (x)_2 = (x)_4$, we see that
\begin{align*}
 \E (D'_1)^2
 \weq \int \Big( (x-2)_2 y^2 + (x-2) y \Big) \frac{(x)_2 y P(dx,dy)}{P_{21}}
 \weq \frac{P_{43} + P_{32}}{P_{21}}.
\end{align*}
Hence $D'_1$ has a finite second moment, and variance equal to 
\begin{equation}
 \label{eq:VarPrime}
 \Var( D'_1 )
 \weq \frac{P_{43} + P_{32}}{P_{21}} - \left( \frac{P_{32}}{P_{21}} \right)^2.
\end{equation}
Similarly, the conditional independence of $D_1'$ and $D'_2$, together with the formula $(x-2)^2 (x)_2 = (x-2) (x)_3 = (x)_4 + (x)_3$, implies that
\begin{align*}
 \E D'_1 D'_2
 \weq \int \big( (x-2)y \big)^2  \frac{(x)_2 y P(dx,dy)}{P_{21}} 
 \weq \frac{P_{43} + P_{33}}{P_{21}},
\end{align*}
and hence, noting that $D'_1$ and $D'_2$ are identically distributed,
\begin{equation}
 \label{eq:CovPrime}
 \begin{aligned}
 \Cov(D'_1, D'_2)
 &\weq \frac{P_{43} + P_{33}}{P_{21}} - \left( \frac{P_{32}}{P_{21}} \right)^2.
 \end{aligned}
\end{equation}
Recall next that $D_1$ follows the compound Poisson distribution $\bar f_1 = \CPoi(\lambda, g)$, and that the variance of $\CPoi(\lambda, g)$ equals $\lambda \int x^2 \, g(dx)$ (\cite{daykin1993practical}, eq. (3.2.13)). A simple computation confirms that the second moment of $g$ in \eqref{eq:LimitingModelDegreeIncrement} equals $ \frac{P_{32} + P_{21} }{P_{10}}$. 
Hence it follows that $D_1$ has a finite second moment with 
\begin{equation}
 \label{eq:VarDeg}
 \Var(D_1)
 \weq \lambda \frac{P_{32} + P_{21} }{P_{10}}.
\end{equation}
The mutual independence of $D_1$, $D_2$, and $(D_1',D_2')$ implies that $\Cov(D_1^*, D_2^*) = \Cov(D'_1, D'_2)$ and $\Var(D_1^*) = \Var(D_1) + \Var(D_1')$, so that
\begin{equation}
 \label{eq:CorSimple}
 \Cor(D_1^*, D_2^*)
 \weq \frac{\Cov(D'_1, D'_2)}{\Var(D_1) + \Var(D_1')}. 
\end{equation}
By plugging \eqref{eq:VarPrime}--\eqref{eq:VarDeg} into \eqref{eq:CorSimple}, we conclude that
\[
 \Cor(D^*_1,D^*_2) 
 \weq \frac{\frac{P_{43} + P_{33}}{P_{21}} - \left( \frac{P_{32}}{P_{21}} \right)^2}
 {\frac{P_{43} + P_{32}}{P_{21}} - \left( \frac{P_{32}}{P_{21}} \right)^2 + \lambda \frac{P_{32} + P_{21} }{P_{10}}}.
\]
By recalling that $\lambda = \mu P_{10}$, the claim follows.
\qed
\end{proof}

\subsection{Proof of Lemma \ref{the:BivariateMomentInequality}}
Let $(X,Q)$ and $(Y,R)$ be mutually independent random vectors, both distributed according to $P$. Then
\begin{align*}
 P_{21} P_{43} + P_{21} P_{33} - P_{32}^2  
 \weq \E f(X,Q,Y,R),
\end{align*}
where
\[
 f(x, q, y, r) \weq (x)_2 q (y)_4 r^3 + (x)_2 q (y)_3 r^3 - (x)_3 q^2 (y)_3 r^2.
\]
On the other hand, by applying the identity $ (y)_4 + (y_3) = (y-2)^2 (y)_2$, we find that
\begin{align*}
 f(x, q, y, r) &\weq (x)_2 (y)_2 (y-2)^2 q r^3 - (x)_3 q^2 (y)_3 r^2 \\
 &\weq (x)_2 (y)_2 q r \Big( (y-2)^2 r^2 - (x-2)(y-2)qr \Big).
\end{align*}
Therefore,
\begin{align*}
 f(x, q, y, r) + f(y, r, x, q) \weq (x)_2 (y)_2 q r \Big( (x-2) q - (y-2) r \Big)^2 \wge 0.
\end{align*}
By symmetry, $\E f(X,Q,Y,R) = \E f(Y,R,X,Q)$, and hence we may conclude that
\[
P_{21} P_{43} + P_{21} P_{33} - P_{32}^2  \weq \frac12 \E \Big( f(X,Q,Y,R) + f(Y,R,X,Q) \Big) \wge 0.
\]
\qed

\subsection{Proof of Theorem \ref{the:ModelAssortativity}}

We only sketch the proof in the case where $\frac{m}{n} \to \mu \in (0,\infty)$. Let $(D_{1,n}^*, D_{2,n}^*)$
 be a random variable distributed according to the model bidegree distribution $f_{2,n}$ of $G_n$. By applying Theorem~\ref{the:LimitingModeBidegree}:\eqref{ite:Wasserstein} with $\phi(x,y) = x$, and then with $\phi(x,y) = x^2$, we find that $\Var(D_{1,n}^*) \to \Var(D_{1}^*)$. Observe next that for $\phi(x,y) = xy$, $\abs{\phi(x,y)} \le 2(x^2+y^2)$. Hence Theorem~\ref{the:LimitingModeBidegree}:\eqref{ite:Wasserstein} also implies that $\Cov(D_{1,n}^*,D_{2,n}^*) \to \Cov(D_{1}^*,D_{2}^*)$. Hence the claim follows by Proposition~\ref{the:LimitingBidegreeCorrelation}.
\qed

\subsection{Proof of Theorem \ref{the:RankCorrelation}}
Because $f_{2,n}$ has identical marginals, we see that
\[
 \rhoken(f_{2,n})
 \weq \frac{\int \phi \, d (f_{2,n} \otimes f_{2,n}) - \left( \int \phi_1 \, d (f^{(1)}_{2,n} \otimes f^{(1)}_{2,n}) \right)^2 }
 {\int \phi_1^2 \, d (f^{(1)}_{2,n} \otimes f^{(1)}_{2,n}) - \left( \int \phi_1 \, d (f^{(1)}_{2,n} \otimes f^{(1)}_{2,n}) \right)^2},
\]
where $\phi_1(x_1,y_1) = \sgn(x_1-y_1)$, $\phi(x_1,x_2,y_1,y_2) = \phi_1(x_1,y_1) \phi_1(x_2,y_2)$ are bounded (and trivially continuous) functions defined on $\Z_+^2$ and $\Z_+^4$, respectively. Theorem \ref{the:LimitingModeBidegree} implies that $f_{2,n} \to \bar f_2 $ weakly as probability measures on $\Z_+^2$. Hence also $f_{2,n} \otimes f_{2,n} \to \bar f_2 \otimes \bar f_2 $ and $f^{(1)}_{2,n} \otimes f^{(1)}_{2,n} \to \bar f^{(1)}_2 \otimes \bar f^{(1)}_2$ weakly.
We conclude that $\rhoken(f_{2,n}) \to \rhoken( \bar f_2 )$.

To verify the claim for Spearman's rank correlation, we apply the representation \cite[Section 4.3]{Neslehova_2007}
\begin{align*}
\rho_{\text{Spe}} (f_{2,n}) \!=\! 
\frac{  \pr( (X^{(1)}\!-\!Y^{(2)})(X^{(2)}\!-\!Z^{(2)}) \!>\! 0) \!-\! \pr((X^{(1)}\!-\!Y^{(2)})(X^{(2)}\!-\!Z^{(2)}) \!<\! 0)   }  { \frac{1}{3}\sqrt{1 - \pr(X^{(1)}=Y^{(1)}=Z^{(1)})} \sqrt{1 - \pr(X^{(2)}=Y^{(2)}=Z^{(2)})} },
\end{align*} 
where $(X^{(1)},X^{(2)}), (Y^{(1)},Y^{(2)}), (Z^{(1)},Z^{(2)})$ are mutually independent and $f_{2,n}$-
distributed. By applying the formula $\pr(W > 0) - \pr(W < 0) = \E [ 1(W>0) - 1(W<0)] = \E \sgn(W) $ with $W = (X^{(1)}-Y^{(2)})(X^{(2)}-Z^{(2)})$, and noting that $f_{2,n}$ has identical marginals, this can be rewritten as 
\[
 \rhospe(f_{2,n})
 \weq 3 \frac{ \int \phi \, d( f_{2,n} \otimes f_{2,n} \otimes f_{2,n})  }
 {\int \psi \, d( f^{(1)}_{2,n} \otimes f^{(1)}_{2,n} \otimes f^{(1)}_{2,n}) },
\]
where $\phi(x_1,x_2,y_1,y_2,z_1,z_2) = \sgn( (x_1-y_2)(x_2-z_2) )$ and $\psi(x_1,y_1,z_1) = 1 - 1(x_1=y_1=z_1)$ are bounded (and trivially continuous) functions on $\Z_+^6$ and $\Z_+^3$, respectively. The second claim follows by noting that $f_{2,n} \otimes f_{2,n} \otimes f_{2,n} \to \bar f_2 \otimes \bar f_2 \otimes \bar f_{2}$ and $f^{(1)}_{2,n} \otimes f^{(1)}_{2,n} \otimes f^{(1)}_{2,n} \to \bar f^{(1)}_2 \otimes \bar f^{(1)}_2 \otimes \bar f^{(1)}_{2}$ weakly.
\qed

\section*{Acknowledgements}
This work was supported by COSTNET COST Action 15109. JK was supported by the Magnus Ehrnrooth Foundation.

\bibliographystyle{splncs04}
\bibliography{lslReferences2}

\end{document}